\documentclass{article}

\usepackage{arxiv}

\usepackage[utf8]{inputenc} 
\usepackage[T1]{fontenc}    
\usepackage{hyperref}       
\usepackage{url}            
\usepackage{booktabs}       
\usepackage{amsfonts}       
\usepackage{nicefrac}       
\usepackage{microtype}      
\usepackage{lipsum}
\usepackage{graphicx}
\usepackage{amssymb,amsthm,amsmath,color,makeidx}

\newcommand{\Pd}{(P) }
 
\newcommand{\inta}{\int_{\R^N}a(x)| u|^q dx}
\newcommand{\intb}{\int_{\R^N}b(x)| u|^p dx}

\newcommand{\HA}{\mbox{$H^1_A (\R^N)$}}

\newcommand{\C}{\mathbb{C}}
\newcommand{\N}{\mathbb{N}}
\newcommand{\R}{{\mathbb R}}
\newcommand{\Z}{\mbox{${Z\!\!\!Z}$}}
\newcommand{\M}{\mathbb{M}}

\newtheorem{theorem}{Theorem}[section]
\newtheorem{corollary}[theorem]{Corolary}
\newtheorem{remark}[theorem]{Remark}
\newtheorem{prop}[theorem]{Proposition}
\newtheorem{lemma}[theorem]{Lemma} 
\newtheorem{definition}[theorem]{Definition}

\title{Existence and multiplicity results for a class of non-linear Schrödinger equations with magnetic potential involving sign-changing non linearity}

\author{
	de Paiva, Francisco Odair Vieira \thanks{F.O.V.P. received research grants from FAPESP 17/16108-6. } \\
	Departamento de Matem\'atica, UFSCar\\
	S\~{a}o Carlos, SP,  13560-970 Brazil\\
	\texttt{franciscoodair@gmail.com} \\
	\And
	de Souza Lima, Sandra Machado \thanks{S.M.S.L. was supported by CAPES/Brazil and the paper was completed while the second author was visiting the Departament of Mathematics of UFJF, whose hospitality she gratefully acknowledges.  }\\
	Departamento de Ciências Exatas, Biológicas e da Terra\\ INFES-UFF\\
	Santo Antônio de Pádua - RJ, 28470-000, Brazil\\
	\texttt{Corresponding author:{sandra.msouzalima@gmail.com}} \\
	\And
	Miyagaki, Olimpio Hiroshi  \thanks{ O. H. M. received research grants from CNPq/Brazil 307061/2018-1, FAPEMIG CEX APQ 00063/15 and INCTMAT/CNPQ/Brazil. }\\
	Departamento de Matemática, UFJF\\
	Juiz de Fora, MG, 36036-900, Brazil\\
	\texttt {ohmiyagaki@gmail.com} \\
}

\begin{document}
	\maketitle
	
	\begin{abstract}
   In this work we consider the following class of elliptic problems
$$
\left\{ \begin{array} [c]{ll}
- \Delta_A u + u = a(x) |u|^{q-2}u+b(x) |u|^{p-2}u , \mbox{  in  } \R^N &\\
u\in\HA,&\\
\end{array}          \right.\leqno {(P) }
$$
with $2<q<p<2^*= \frac{2N}{N-2}$, $a(x)$ and $b(x)$ are functions that can change signal and satisfy some additional conditions; $u \in H^1_A(\R^N)$ and $A:\R^N \rightarrow\R^N$ is a magnetic potential. Also using the Nehari method in combination with other complementary arguments, we discuss the existence of infinite solutions to the problem in question, varying the assumptions about the weight functions.
	\end{abstract}


  
\section{Introduction}

We are interested in studying the following class of elliptical problems 
$$
\left\{ \begin{array} [c]{ll}
- \Delta_A u + u = a(x) |u|^{q-2}u+b(x) |u|^{p-2}u , \mbox{  in } \R^N &\\
u\in\HA,&\\
\end{array}          \right.\leqno {(P) }
$$
with $2<q<p<2^*= \frac{2N}{N-2}$, $a(x)$ and $b(x)$ are functions that can change signal and satisfy some additional conditions. $u \in H^1_A(\R^N)$ and $A:\R^N \rightarrow\R^N$ is a magnetic potential. We will discuss the existence of infinite solutions to the problem in question, varying the hypotheses about the weight functions.

We will make use of the magnetic operator in which we work with the Magnetic Laplacian given by $(i\nabla -A)^2 + V$, where $A: \R^N \rightarrow \R^N$ is the magnetic potential and $ V: \R^N \rightarrow \R $ is the electrical potential. Its importance in physics was discussed in Alves and Figueiredo \cite{ClauFig} and also in Arioli and Szulkin\cite{ArSz}.

Still seeking to contextualize the problems that we deal with in this work, we will speak a little of what has been done with respect to problems of the convex type with the usual Laplacian. Starting with the work of Alama and Tarantello \cite{1} who consider the problem
$$
\left\{ \begin{array} [c]{ll}
- \Delta u - \lambda u = W (x)f (u), x \in \Omega &\\
u(x) = 0,\;\; x \in \partial  \Omega.&\\
\end{array}          \right.\leqno { }
$$
In this case $W$ is a function that can change sign and $ f $ satisfies $\lim_{ |u|\rightarrow \infty} f (u)/|u|^{p- 2} u = a > 0$ for some $2 < p \leq  2^*  (2 < p < \infty $ if $ N = 1 $ or $ 2)$. They deal with the existence of positive solutions. In the case where $ f $ is an odd function, they show the existence of infinite solutions, which may even be solutions that change sign.

In \cite{4}, Berestycki et al. study the existence and non-existence of solutions to the following class of problems
$$
\left\{ \begin{array} [c]{ll}
- \Delta u + m(x)u = a(x)u^p , x \in \Omega, \\
Bu(x) = 0, x \in \partial \Omega. &\\
\end{array}          \right.\leqno { }
$$
In this case, $m(x)$ can change signal, $1 < p < 2^*  -  1$ and $Bu = u$, $Bu = \partial_{ \nu}  u$ or $Bu = \partial_{ \nu}  u + \alpha (x)u, \alpha  > 0$.

Other works that also dealt with the convex case in the bounded domain with the usual Laplacian can be seen in \cite{2,BW,BZ2003,12,19}. 

Already in $ \R^N, $ Miyagaki \cite{olimpio1} studies the existence of nontrivial solutions for the following class of elliptical problems
$$
- \Delta u + a(x)u = \lambda|u|^{q-1}u+  |u|^{p-1}u , \; x\;\; \in \;\; \R^N,$$
where  $1<p<q\leq 2^*-1= \frac{N+2}{N-2}$, $\lambda>0$ is a constant and $a(x):\R^N \rightarrow \R$ is a continuous function satisfying some additional conditions.

Jalilian and Szulkin \cite{jasz} an equation of the type
$$
\left\{ \begin{array} [c]{ll}
- \Delta u + u = a(x) |u|^{p-2}u+b(x) |u|^{q-2}u ,  \; x\;\; \in \;\; \R^N, &\\
u\in H^1(\R^N),&\\
\end{array}          \right.\leqno { }
$$
with $2<p<q<2^*= \frac{2N}{N-2}$, on what $a(x)$ or $b(x)$ are functions that can change signal. They investigate the existence of infinite solutions. This was the problem we tried to extend.

We will present, at this point, the main works that use the magnetic laplacian, in order to contextualize and highlight the importance of the study that we develop.

The first results in non-linear Schr\"{o}dinger equations, with  $A \neq 0 $ can be attributed to Esteban and Lions \cite{EstLions} in which the existence of stationary solutions for equations of the type
$$-\Delta_A+Vu =|u|^{p-2}u, u\neq 0, u\in L^2(\R^N)  , $$
$p \in (2,\infty),$  using minimization methods for the case $ V = 1 $, with constant magnetic field and also for the general case.

In \cite{Kurata}, Kurata showed that the equation
\begin{equation}\label{eq.kurata}
\left(\frac{h}{i}\nabla -A(x)\right)^2u+V(x)u - f(|u|^2)u=0, x\in \R^N
\end{equation}
with certain assumptions about the magnetic field A, as well as for the potential V and f, has at least
a solution that focuses near the set of global minimums of V, as $h\rightarrow 0$.

Also, Chabrowski and Szulkin \cite{ChabSzul} worked with this operator in the critical case and with the electric potential V being able to change signal. Cingolani, Jeanjean and Secchi \cite{Cingolani2} consider the existence of mult-peak solutions in the subcritical case, obtaining multiplicity results using topological arguments and proving that the magnetic potential A only contributes to the solution phase factor of the equation (\ref{eq.kurata}) for very small values of $h>0$. 

Alves and Figueiredo \cite{ClauFig} work with a problem of the type
$$-\Delta_A u = \mu |u|^{q-2}u+|u|^{2^*-2}u, u \neq  0, x \in \Omega \subset \R^N ,  $$
 $\Omega$ is a bounded domain, $\mu>0$ and $2\leq q <2^*$, which relates the number of solutions with the topology of $\Omega$.

Also, Alves, Figueiredo and Furtado \cite{alves.furtado} studied the following equation
$$\left( -i\nabla -A \left( \frac{x}{\lambda}\right) \right) ^2 u +u=  f(|u|^{2})u , x \in \Omega_{\lambda},$$
in which the set $\Omega \subset \R^N $ is a bounded domain, $\lambda> 0$ is a real parameter, $A$ is a regular magnetic field and $f$ is a superlinear function with subcritical growth. For values of $\lambda$ sufficiently large, the authors show the existence and multiplicity of solutions relating the number of solutions with the topology of $\Omega$, 

In our case $A:\R^N \rightarrow\R^N$ is a magnetic potential in $L^2_{loc}(\R^N,\R^N)$. We will use the method introduced by Nehari in 1960, which has become very useful in critical-point theory. The method of fibering map introduced by Drabek and Pohozaev \cite{DP} and discussed by Brown and Zhang \cite{BZ2003}, relates the functional to a real function. The information about this function gives us an important basis to achieve the result we are looking for. We did not find in the literature works with the Laplacian magnetic that deal with the convex case with weight functions and can change the signal. We are interested in studying the following class of elliptical problems
$$
\left\{ \begin{array} [c]{ll}
- \Delta_A u + u = a(x) |u|^{q-2}u+b(x) |u|^{p-2}u ,  &\\
u\in\HA,&\\
\end{array}          \right.\leqno {(P) }
$$
where $x \in \R^N$, $2<q<p<2^*= \frac{2N}{N-2}$, $a$ and $b$ are functions that can change signal and satisfy some additional conditions. Besides that, $u \in H^1_A(\R^N)$ and $A:\R^N \rightarrow\R^N$ is a magnetic potential $L^2_{loc}(\R^N,\C)$. We will discuss the existence of solutions and, varying the hypotheses under the weight functions, we will study the existence of infinite solutions to the problem in question.

In this case, we need to enumerate some hypotheses among which the functions $ a, b $ can assume in the theorems that follow. Before that, we will make the following definition.

\begin{definition}
	Let $g \in E(\R^N)$ and $j  \in \Z$. We say that the function $ g $ is $ 1- $ periodic in $ x_i $ if
	$$ g(x_1,x_2, ... , x_i, ..., x_n)= g(x_1,x_2, ... , x_i+j, ..., x_n),$$
	for $i=1,2,...,n$.
\end{definition}

In what follows, we assume that $ a, b $ satisfies some of the following hypotheses:

\begin{enumerate}
	
	\item[($D_1$)] $a\in L^r$ and $b\in L^s$, where $1<\frac{r}{r-1}<\frac{2^*}{p}$ and $ 1<\frac{s}{s-1}<\frac{2^*}{q}$;
	
	\item[($D_2$)] $a, b \in L^{\infty}(\R^N)$, $\limsup_{|x|\rightarrow \infty}a(x) \leq 0$ and $\limsup_{|x|\rightarrow \infty}b(x) \leq 0$;
	
	\item[($D_3$)] $a, b \in L^{\infty}(\R^N)$, $a$ and $b$ are 1-periodic functions in $x_1$, $x_2$,...,$x_n$;
	
	\item[($D_4$)] $b\geq 0$ and the set $\{x\in \R^N ; b>0\}$ has not empty interior;
	
	\item[($D_5$)] $a\leq 0$ and the set $\{x\in \R^N ; b>0\}$ has not empty interior.
\end{enumerate}

The conditions ($ D_4 $) and ($ D_5 $) appear first in \cite[Example 4.3]{12} in which the existence of positive solutions to a problem with the usual Laplacian and bounded domain is studied. Also Jalilian and Szulkin in \cite{jasz}, use the above hypotheses to treat an elliptical problem in $\R^N$.



\begin{theorem}\label{teo2.1}
Assume that ($ D_1 $) or ($ D_2 $) and ($ D_4 $) or ($ D_5 $) is satisfied. So the problem $ \Pd $ has infinite solutions.
\end{theorem}

We will now announce our second result.

\begin{theorem}\label{teo2.2}
	Assume $ (D_3) $ and also $ (D_4) $ or $ (D_5) $ are satisfied. So the problem $ \Pd $ has infinite geometrically distinct solutions.
\end{theorem}

We will use the relationship between the Nehari manifold and the fibering map to discuss the existence of nontrivial solutions for this class of elliptical problems. We will show that the Nehari manifold is closed and is a $ C ^ 2 $ manifold under each of the hypotheses of the above theorems. To obtain the result of Theorem \ref{teo2.1}, we show that the PS condition is satisfied in the manifold and we will use a Krasnoselskii genus argument, which can also be seen in \cite{jasz}. Under the hypotheses Teorema \ref{teo2.2}, the PS condition is not valid, hence we need to use an  deformation type argument which is based on an idea of Szulkin and Weth \cite{16} to show the existence of infinitely different geometric solutions. \\

\section{ Initial Considerations for the Problem $\Pd$}

In this section we will define the Nehari manifold associated with the $ \Pd $ problem and its relation to the fibering map. The functional associated with the problem in question is given by
$$ I(u) :=\frac{1}{2}||u||_A^2 - \frac{1}{q}\int_{R^N}a(x)|u|^q dx - \frac{1}{p}\int_{R^N}b(x)|u|^p dx$$
and since ($D_1$), ($D_2$) or ($D_3$) are satisfied we will have that the functional is of class $C^2 (\HA,\C$).
We can also see that the critical points of the functional are weak solutions of the problem $\Pd$.
\begin{prop} If ($D_4$) or ($D_5$) are satisfied, then the functional $ I $ is not bounded below in $\HA$.
\end{prop}

\begin{proof} Considering without loss of generality that $\int_{\R^N } b(x)|u|^{p}dx > 0$, for some $u \in \HA$ with $u > 0$ in $\R^N$ and consider $t > 0$, like this
	{\small
		\begin{eqnarray*}
			I (tu) 	&=&\frac{t^2}{2}||u||^2_A -\frac{ t^{q}}{q }\int_{\R^N}a(x)|u|^{q}dx - \frac{t^{p}}{p}\int_{\R^N}b(x)|u|^{p}dx\\
			&=&t^{p}\!\left(\frac{1}{2 t^{p-2}}||u||^2_A\! -\frac{1}{qt^{p-q}}\!\int_{\R^N }a(x)|u|^{q}dx-\frac{1}{p}\! \int_{\R^N }b(x)|u|^{p}dx\right).
	\end{eqnarray*}}
	Taking $t \rightarrow \infty$,  as $ 2 < q < p < 2^*$, we have that $I (tu) \rightarrow -\infty$, that is, $I $ is not bounded below in $\HA.$ 
\end{proof}

\subsection{Nehari Manifold associated with $\Pd$}
We want to find a subset of $ \HA $, where the functional $ I $ is well behaved, that is, where this function is bounded below. We then define
$$\M=\{u \in \HA\setminus\{0\} :\langle I'(u),u\rangle =0\}.$$ 
$\M $ is the Nehari manifold associated with functional $I$. Therefore,
$$ u\in \M \Leftrightarrow I' (u)u=0 \Leftrightarrow $$
\begin{equation}\label{u.in.nehari}
||u||_A^2 -  \int_{\R^N } a(x) |u|^{q } dx - \int _{\R^N } b(x) |u|^{p}dx=0 .
\end{equation} 
We note that $ \M \subset \HA $ and now we will find the functional set on the Nehari manifold. For $ u \in \M$
\begin{eqnarray*}
	I (u) & = &  I (u)- \frac{1}{q}I'(u)u\\
	& = &   \left(  \frac{1}{2} - \frac{1}{q } \right) ||u||_A^2-\left(  \frac{1}{p }- \frac{1}{q } \right)  \int_{\R^N }b(x)|u|^{p }dx;\\
	& = &   \left(  \frac{1}{2} - \frac{1}{p } \right) ||u||_A^2-\left(  \frac{1}{q }- \frac{1}{p } \right)  \int_{\R^N }a(x)|u|^{q }dx.
\end{eqnarray*}
We will now see that the $ I$ functional is well behaved in the Nehari manifold.

\begin{corollary}\label{cor2.3}
	Assume that ($ D_4 $) or ($ D_5 $) is satisfied, then the functional $ I $ is bounded from below in $\M.$
\end{corollary}

\begin{proof} In fact, by (\ref {u.in.nehari}) and by ($ D_4 $) we have
	\begin{eqnarray} \label{ltdo2}
	\nonumber I(u)&=& \frac{1}{2}	||u||^2_A -\frac{1}{q} \int_{\R^N } a(x) |u|^{q } dx -\frac{1}{p} \int _{\R^N } b(x) |u|^{p}dx \\
	\nonumber &= &\left( \frac{1}{2}-\frac{1}{q}\right) 	||u||^2_A-\left( \frac{1}{p}- \frac{1}{q}\right) \int_{\R^N } b(x) |u|^{p } dx \geq 0,
	\end{eqnarray}
	yet, when ($ D_5 $) is satisfied
	\begin{eqnarray} \label{ltdo1}
	\nonumber I(u)&=& \frac{1}{2}	||u||^2_A -\frac{1}{q} \int_{\R^N } a(x) |u|^{q } dx -\frac{1}{p} \int _{\R^N } b(x) |u|^{p}dx \\
	\nonumber &= &\left( \frac{1}{2}-\frac{1}{p}\right) 	||u||^2_A-\left( \frac{1}{q}- \frac{1}{p}\right) \int_{\R^N } a(x) |u|^{q } dx \geq 0,
	\end{eqnarray}
	which concludes the boundness from bellow.
	
\end{proof}

We will now make some considerations and we will present some properties of the manifold and its relation with the fibering map.

\subsection{Fibering Map}

We will now define the fibering map associated with the functional $ I $, which are the functions of the form $ T_u: t \rightarrow I (tu); \; \ (t> 0) $, we will analyze its behavior and show its relation with the manifold of Nehari.

If $u \in  \HA$, we have
\begin{equation}\label{T}
T _u(t)=\frac{t^2}{2} ||u||_A^2  -\frac{t^{q}}{q}\inta - \frac{t^{p}}{p}\intb,
\end{equation}
\begin{equation}\label{T'}
T'_u(t)= t ||u||_A^2 -  t^{q-1} \inta -  t^{p-1} \intb,
\end{equation}
\begin{equation}\label{T''}
T''_u(t)= ||u||_A^2 - (q-1) t^{q-2}  \inta - (p-1) t^{p-2} \intb.
\end{equation}

The proposition below relates the Nehari manifold and the Fibering map.

\begin{prop}
Let $ T_u $ be the application defined above and $ u \in \HA $, then
	\begin{description}
		\item[	$(i)$   ] $ u \in \M $ if and only if, $T_u'(1)=0$;
		\item[	$(ii)$   ] More generally $tu \in \M $ if and only if, $T'_u (t)=0$. 
	\end{description}
	\end{prop}


From the definitions made, we will analyze the behavior of the application fiber in order to obtain information about our functional.

\begin{remark}
	Note that if $ u \in \M $, that is, $ T'_u (1) = 0 $, then
	\begin{eqnarray}\label{e0}
	T'' _u(1)&=& (2-q) ||u||_A^2  - (p-q) \int_{\R^N } b(x)|u|^{p }dx  \\
	&=& (2-p) ||u||_A^2  - (q-p) \int_{\R^N } a(x)|u|^{q }dx.
	\end{eqnarray}
\end{remark}

Note that the essential nature of the fibering map $ T_u $ is determined by the signal of $ \int_{\R^N } a(x)|u|^{q }dx$ and $ \int_{\R^N } b(x)|u|^{p }dx $. In fact, consider the function
\begin{equation}\label{e1}
m_u(t)=\frac{1}{ t^{q-2}} ||u||_A^2  - t^{p-q}\int_{\R^N } b(x)|u|^{p }dx; \;\;\;t>0
\end{equation}

\begin{remark}
	Note that, for $t>0$, $tu \in \M $ if and only if $t$ is a solution of
	\begin{equation}\label{e2}
	m_u(t)=\int_{\R^N }a(x)|u|^{q }dx.
	\end{equation}
\end{remark}

In fact, by replacing (\ref{e1}) in (\ref{e2}), we have
$$   \int_{\R^N }a(x)|u|^{q }dx=t^{2-q} ||u||_A^2  - t^{p-q}\int_{\R^N } b(x)|u|^{p }dx,$$
$$0=t^{2-q} ||u||_A^2 -  \int_{\R^N }a(x)|u|^{q }dx- t^{p-q}\int_{\R^N } b(x)|u|^{p }dx.$$
Multiplying the above equation by $t^{q }$
$$0=t^2 ||u||_A^2 -   t^{q }\int_{\R^N }a(x)|u|^{q }dx- t^{p }\int_{\R^N } b(x)|u|^{p }dx,$$
or equivalently $I'(tu)tu=0.$ Therefore, $tu \in \M.$ Also, deriving (\ref{e1}) we get
\begin{equation}\label{e3}
m'_u(t)=(2-q)t^{1-q} ||u||_A^2  - (p-q)t^{p-q-1}\int_{\R^N } b(x)|u|^{p }dx.
\end{equation}

Let us now analyze the behavior of $ m_u $ for the following cases.
\begin{description}
	\item[$(i)$   ]  When $\int_{\R^N } b(x)|u|^{p }dx> 0$, $m_u$ is a strictly decreasing function.   
\end{description}
In fact, where $2<q<p$ for $t>0$
$$m'_u(t)=(2-q)t^{1-q} ||u||_A^2  - (p-q)t^{p-q-1}\int_{\R^N } b(x)|u|^{p }dx  <0,$$
whenever $\int_{\R^N } b(x)|u|^{p }dx> 0 $.
In addition, if $t \rightarrow0$ then $m_u(t)\rightarrow +\infty.$ Now, if $t \rightarrow \infty$ then $m_u(t)\rightarrow -\infty.$ Thus we conclude that $m_u(t)$ has a single point of inflection in $t_i=\left( \frac{(2-q)||u||^2_A}{(p-q)\int_{\R^N } b(x)|u|^{p }dx} \right)^{\frac{1}{p-2}}<0,$ and its graph has a sketch as in the figure \ref{fig1}.

\begin{figure}[h]
		\centering
		\includegraphics[scale=0.4]{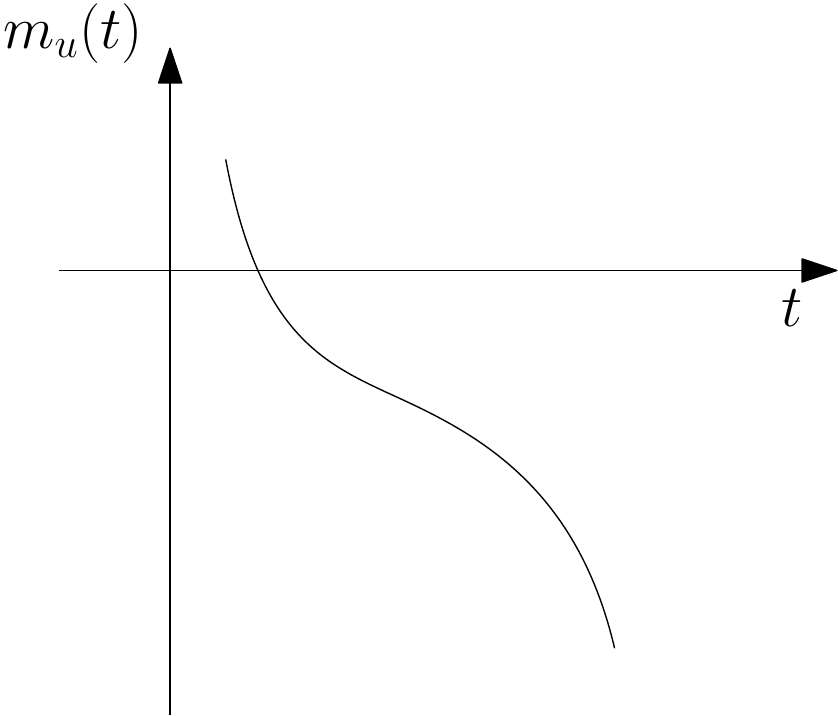}
	\caption{ $m_u$ sketch when $\int_{\R^N }b(x)|u|^{p }dx > 0 $}
	\label{fig1}
\end{figure}

\begin{description}
	\item[$(ii)$   ] When $\int_{\R^N } b(x)|u|^{p }dx= 0$, $m_u$ is also a strictly decreasing function.    
\end{description}

In fact, for $ t> 0 $ then
$$m'_u(t)=(2-q)t^{1-q} ||u||_A^2  <0,$$
whenever $\int_{\R^N } b(x)|u|^{p }dx= 0 $.
Furthermore, if $t \rightarrow0$ then $m_u(t)\rightarrow +\infty.$ If $t \rightarrow \infty$, then $m_u(t)\rightarrow 0.$ In this way we conclude that $ m_u (t) $ has the graph as in the figure \ref{fig2}.

\begin{figure}[h]
	\begin{center}
		\includegraphics[scale=0.3]{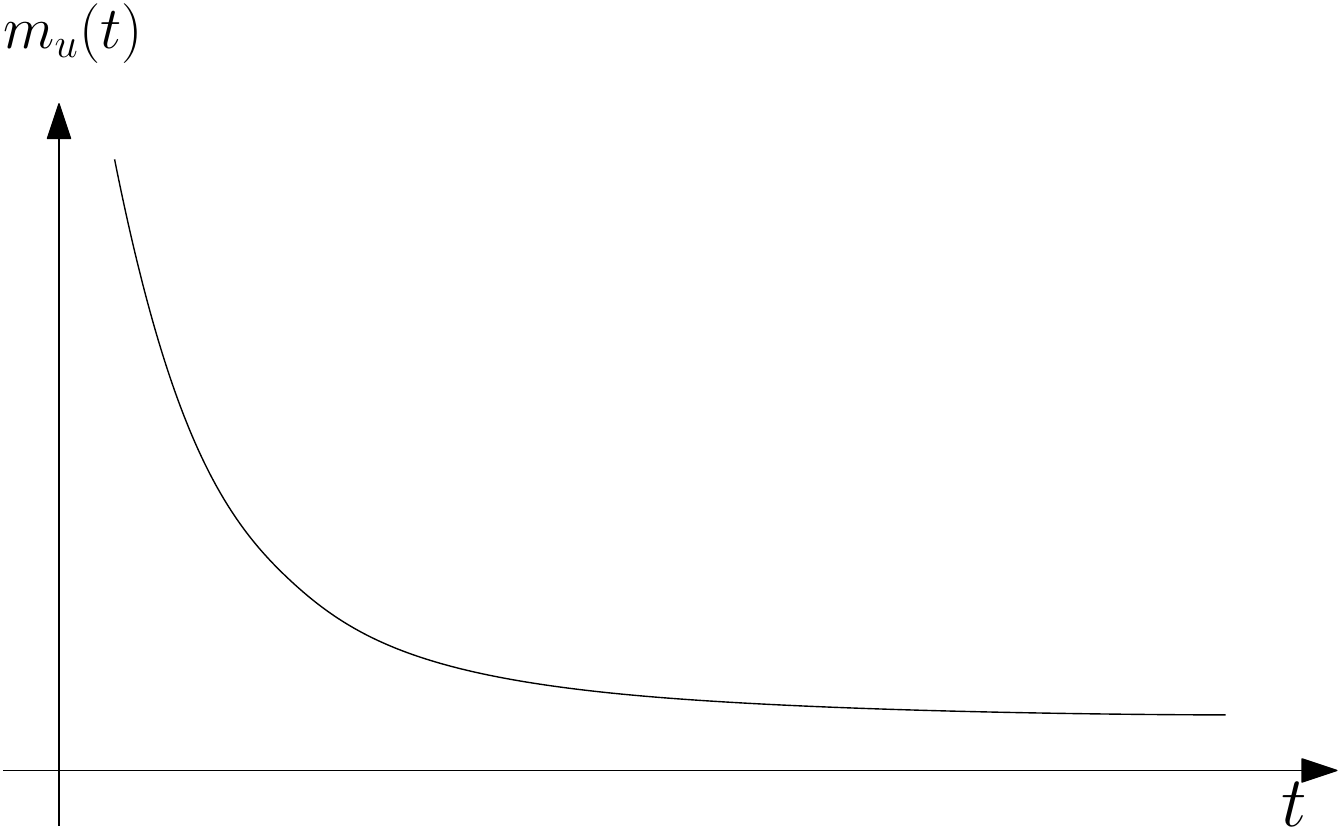}
	\end{center}
	\caption{$ m_u $ sketch when $\int_{\R^N }b(x)|u|^{p }dx = 0 $}
	\label{fig2}
\end{figure}

\begin{description}
	\item[$(iii)$   ]   When $\int_{\R^N } b(x)|u|^{p }dx< 0$. 
\end{description}
In this case $m_u$ is a decreasing and then increasing function with a single critical point in $t_{\min}=\left( \frac{(2-q)||u||^2_A}{(p-q)\int_{\R^N } b(x)|u|^{p }dx} \right)^{\frac{1}{p-2}}.$ In addition, $m_u(t)>0$ for all $t>0$. 
Noting that
$$ \lim_{t \rightarrow 0} m_u(t)= \infty \;\;\mbox {     and    }\;\;
\lim_{t \rightarrow \infty} m_u(t)= \infty,$$
we can conclude that $ m_u $ has a graph as in the figure \ref{fig3}.

\begin{figure}[h]
	\begin{center}
		\includegraphics[scale=0.25]{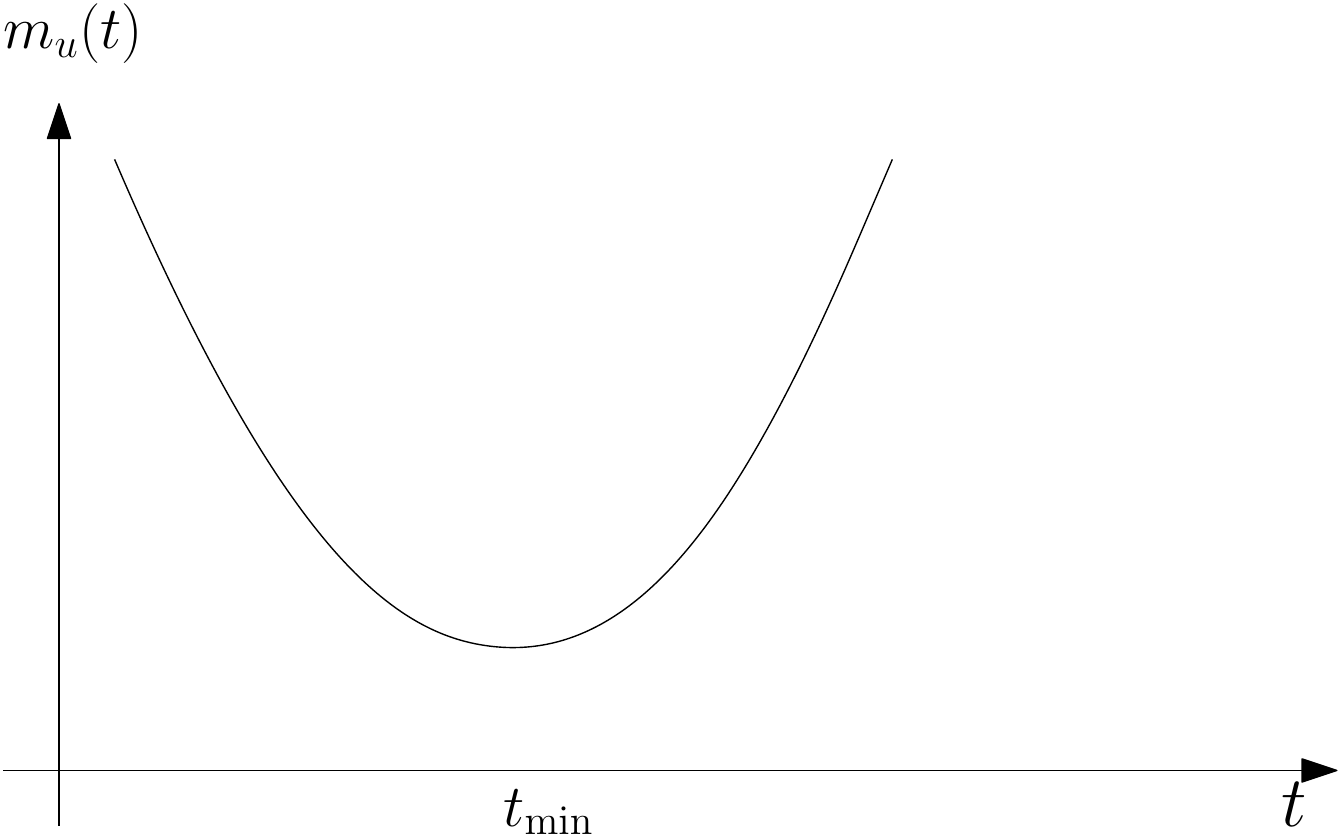}
	\end{center}
	\caption{$ m_u $ sketch when $\int_{\R^N }b(x)|u|^{p }dx <0 $}
	\label{fig3}
\end{figure}

\begin{remark} It is important to note that if $tu \in \M $, by (\ref{e0}) and (\ref{e1}) we have
	$$T''_{tu}(1)=t^{q+1}m'_u(t).$$
\end{remark}

In fact,
\begin{eqnarray*}
	T'' _{tu}(1) & = & (2-q)t^2 ||u||_A^2  - (p-q) t^{p} \int_{\R^N } b(x)|u|^{p }dx  \\
	& = & t^{q+1}((2-q)t^{1-q} ||u||_A^2  - (p-q) t^{p-q-1}\int_{\R^N } b(x)|u|^{p }dx ) \\
	& = & t^{q+1}m'_u(t).
\end{eqnarray*}
This observation is fundamental, because if we know the signal of $ m'_u (t) $, we will know the signal of $ T''_{tu}(t) $. Thus we can know if $ T_{tu} $ has a local minimum, maximum local or inflection point.

\subsubsection{ Function description $ T_u $}
Let us now see the description of the nature of the fibering map for cases where $ (D_4 $) or $ (D_5) $ is satisfied. 
\begin{description}
	\item[$(I)$   ]    When $(D_4)$ is satisfied, there will be  $u's \in \HA$, such that $\int_{\R^N } b(x)|u|^{p }dx>0$ with strict inequality. Looking at the graph that we construct in item $ (i) $ above, there are $ t_u's $, solution of (\ref{e2}) for any value of $ \int_{\R^N } a(x)|u|^{q }dx$.In these conditions, for each $ u $ such that $\int_{\R^N } b(x)|u|^{p }dx>0$, there exists a unique $ t_u> 0 $, such that $t_uu \in \M.$ Also, $t_u>0$ is a maximum point for $ T_ {tu} $, since $T''_{tu}(1)=t^{q+1}m'_u(t)<0$ in this case.From this analysis, we conclude that the graph $ T_u $ has its sketch as shown in the figure \ref{fig4}.
\end{description}

\begin{figure}[h]
	\begin{center}
		\includegraphics[scale=0.4]{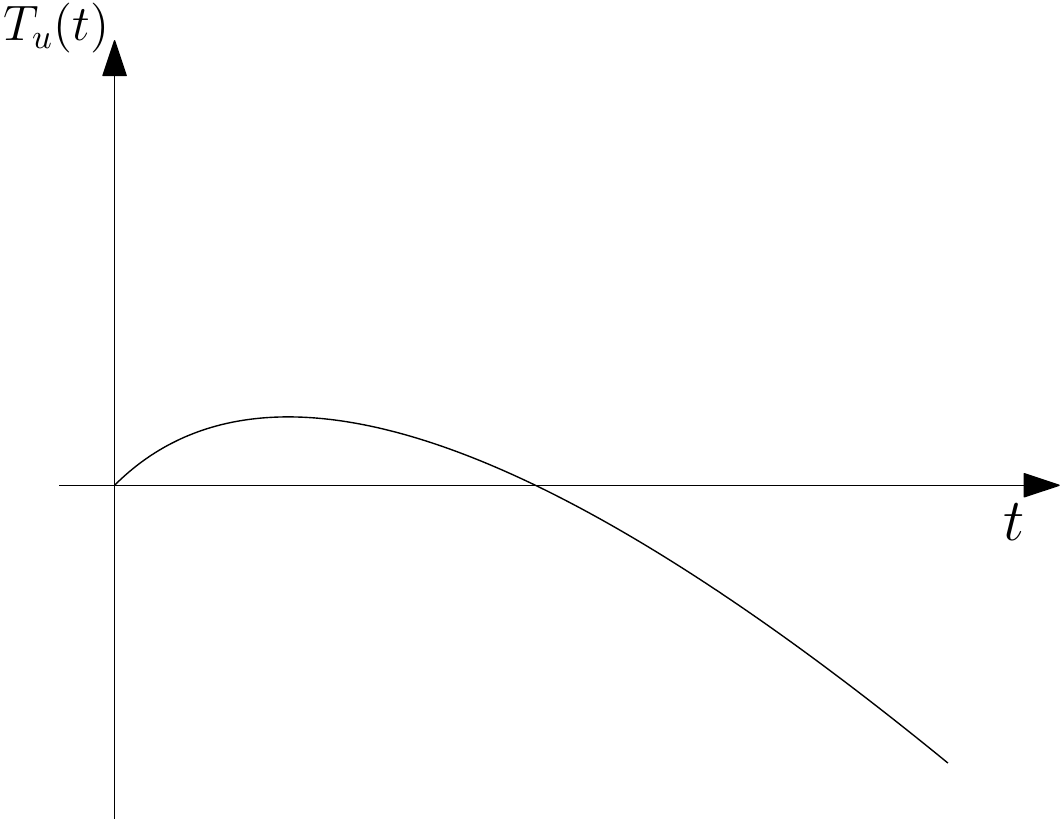}
	\end{center}
	\caption{Possible form of $ T_u $ when $\int_{\R^N }b(x)|u|^{p }dx > 0$}
	\label{fig4}
\end{figure}

\begin{description}
	\item[ $(II)$  ] In the case where the hypothesis $ (D_5) $ is satisfied we have $ \int_{\R^N } a(x)|u|^{q }dx\leq0$. Observing the graphs \ref{fig1}, \ref{fig2} e \ref{fig3}, We see that equation ($\ref{e2}$) only has solution when $\int_{\R^N } b(x)|u|^{p }dx>0.$ Conditions that have already been analyzed in item $(I)$.    
\end{description}
 
From these observations we have the following.
\begin{lemma}\label{lema2.2}
	Suppose the hypothesis ($ D_1 $) is satisfied and that $a$, $b \; \in L^{\infty}(\R^N)$.
	\begin{description}
		\item[ $(i)$  ] If ($D_4$) is satisfied and we also have $\int_{\R^N } b(x)|u|^{p }dx > 0$ or $\int_{\R^N } a(x)|u|^{q }dx > 0$, then the equation $ T'_u (t) = 0 $ has exactly one solution $t_u > 0$. Also, $I(u) > 0$ for all $u \in \M  $;   
		\item[  $(ii)$ ]  If ($D_5$) is satisfied and $\int_{\R^N } b(x)|u|^{p }dx > 0$, then the equation $ T'_u (t) = 0 $ has exactly one solution $t_u > 0$. Also, $I(u) > 0$ for all $u \in \M  $.  
	\end{description}

\end{lemma}


\subsection{Properties of the Nehari manifold}
Next we will see that under certain assumptions the Nehari manifold is indeed a manifold.

\begin{lemma}\label{lema2.3}
	
Suppose that ($ D_1 $) and also ($ D_4 $) or ($ D_5 $) is satisfied. Then the Nehari manifold is a $ C^2 $ manifold, closed and such that $ || u ||_A \geq \delta> 0 $, for all $u \in \M $.
\end{lemma}

\begin{proof}
Let $ u \in \M $, a direct consequence of the definition of the manifold gives us
	\begin{equation*}
	\nonumber	||u||^2_A  = \int_{\R^N } a(x) |u|^{q } dx + \int _{\R^N } b(x) |u|^{p}dx.
	\end{equation*}
	Hence, by H\"{o}lder and Sobolev, by $ (D_1) $ and by the diamagnetic inequality we have
	\begin{eqnarray}
	\nonumber	||u||^2_A  &\leq & ||a||_r ||u||^q_{qr'} +  ||b||_s ||u||^p_{ps'}\\
	\label{2.5} 	& \leq&  c_1 ||a||_r ||u||^q_A + c_2 ||b||_s ||u||^p_A, 
	\end{eqnarray}
	where $r' = \frac{r}{r-1}$ , $ s' = \frac{s}{s-1}$ and $c_1$, $c_2$ are positive constants. Dividing (\ref{2.5}) by $ || u ||^2_A $ we get
	\begin{equation}\label{2.5.2}
	1 \leq c_1 ||a||_r ||u||^{q-2}_A + c_2 ||b||_s ||u||^{p-2}_A.
	\end{equation}
	Assume by contradiction that there exists a sequence $ \{u_n \} \in \M $ such that $ || u_n ||_A \rightarrow 0, $ as $ n \rightarrow \infty$. Then, like $ 2 <q <p $, by (\ref{2.5.2}) we get that $ 1 \leq 0 $ which is absurd. It follows that there is $\delta>0$ such that $||u||_A\geq \delta>0$, for all $u \in \M $.
	
	We will now show that the Nehari manifold is closed and $ C^2$. 
	Define $ \alpha   : X \rightarrow  \R$ by 
	$$ \alpha  (u) := \langle  J' (u) , u  \rangle = ||u||^2_A  - \int_{\R^N } a(x) |u|^{q } dx - \int _{\R^N } b(x) |u|^{p}dx  .$$
	See that $ \alpha \in C^2 $ and by the definition of $ \alpha $, we have to $ \M   = \alpha^{- 1}(0) \setminus \{0\}.$ 
	The fact that $ \delta> 0 $ exists such that $ || u ||_A \geq \delta> 0 $, for every $ u \in \M $, gives $ \M $ is closed.    
	We need to show that $ 0 $ is a regular value of $ \alpha $, that is, for all $u \in \M$, $ \alpha'(u)$ is a linear overhead transformation. Since the application image $\alpha $ is $\R$,that is, it is a space of one dimension, just show that $ \alpha' (u) \neq 0$ for all $u \in \M$.
	Note that for every $u \in \M$
	\begin{equation}\label{2.6i}
	||u||_A^2= \int_{\R^N } a(x) |u|^{q } dx + \int _{\R^N } b(x) |u|^{p}dx.
	\end{equation}
	Also,
	\begin{equation}\label{2.7}
	\langle\alpha'(u),u \rangle = 2||u||_A^2-q \int_{\R^N } a(x) |u|^{q } dx -p \int _{\R^N } b(x) |u|^{p}dx.
	\end{equation} 
	Now, by (\ref{2.6i}) and (\ref{2.7}) we have \begin{equation}\label{2.8}
	\langle\alpha'(u),u \rangle = (2-q)||u||_A^2+ (q-p)\int _{\R^N } b(x) |u|^{p}dx.
	\end{equation}
	If ($D_4$) is satisfied then $\int _{\R^N } b(x) |u|^{p}dx \geq 0$ and by (\ref{2.8})
	\begin{equation}\label{2.8.i}
	\langle\alpha'(u),u \rangle <0.
	\end{equation}
	Also, by (\ref{2.6i}) and (\ref{2.7}) we have \begin{equation}\label{2.8.1}
	\langle\alpha'(u),u \rangle = (2-p)||u||_A^2+ (p-q)\int _{\R^N } a(x) |u|^{q}dx.
	\end{equation}
	Hence, if ($ D_5 $) is satisfied then $\int _{\R^N } a(x) |u|^{p}dx \leq 0$ and by (\ref{2.8.1})
	\begin{equation}\label{2.8.ii}
	\langle\alpha'(u),u \rangle <0.
	\end{equation}
Thus, by (\ref{2.8.i}) and (\ref{2.8.ii}) follows that $ \alpha'(u) \neq 0 $. We conclude that $ 0 $ is a regular value of $ \alpha $, giving us that the Nehari manifold is a fact manifold and $ C^2 $ class.
	
\end{proof}

We will now show that the same conclusions from the previous lemma are also valid when we consider the hypotheses described in the following lemma.

\begin{lemma}\label{lema2.4}
	Assume that $a, b \in L^{\infty} (\R^N   )$ and also ($D_4$) or ($D_5$) are satisfied. Then the Nehari manifold is a $C^2$ manifold, closed and such that $||u||_A\geq \delta>0$, for all $u \in \M $
\end{lemma}

\begin{proof}
	Let $u \in \M $ and let $a, b \in L^{\infty} (\R^N   )$. We have that
	\begin{eqnarray} \label{2.9}
	\nonumber	||u||^2_A  &=& \int_{\R^N } a(x) |u|^{q } dx + \int _{\R^N } b(x) |u|^{p}dx \\
	\nonumber &\leq & ||a||_{\infty} ||u||^q_{q} +  ||b||_{\infty} ||u||^p_{p}\\
	& \leq&  c_1 ||a||_{\infty} ||u||^q_A + c_2 ||b||_{\infty} ||u||^p_A, 
	\end{eqnarray}
	where $c_1$, $c_2$ are positive constants. Dividing (\ref{2.5}) by $ || u ||^2_A $ we get
	\begin{equation}\label{2.9.2}
	1 \leq c_1 ||a||_{\infty} ||u||^q_A + c_2 ||b||_{\infty} ||u||^p_A.
	\end{equation}
	Assume by contradiction that there exists a sequence $ \{u_n \} \in \M $ such that $ || u_n ||_A \rightarrow 0, $ as $ n \rightarrow \infty$. Then, like $ 2 <q <p $, by (\ref{2.9.2}) we get that $ 1 \leq 0 $ which is absurd. It follows that there exists $ \delta> 0 $ such that $ || u ||_A \geq \delta> 0 $, for all $u \in \M $.
	The demonstration follows as in Lemma \ref{lema2.3}. We conclude that under these conditions the Nehari manifold is a closed and $C^2$ manifold.
\end{proof}

From the results we have just done we are ready to relate the critical points of the functional restricted to the Nehari manifold with the critical points of the functional defined throughout $\HA$.

\begin{lemma}\label{lema2.5}
	Suppose that the hypotheses of Lemmas \ref{lema2.3} or \ref{lema2.4} are satisfied. Then $ u \neq 0 $ is a critical point of the functional $ I $, if and only if, it is also a critical point of $I|_{\M }$. Also, $\{u_n  \} \subset \M $ is a $(PS)_c$ sequence of $I$ if and only if it is a $ (PS)_c $ sequence for $I|_{\M }$.
\end{lemma}

\begin{proof}
	If $u \neq  0$ is a critical point of $ I $, we have to $\langle I'(u),v \rangle =0 $ for all $v\in \HA$, in particular for $ u = v $ and by the definition of the Nehari manifold, we have that $ u \in \M $. Besides that, $\langle I'(u),v \rangle =0 $ for all $v\in T_u\M$, from where we conclude that $ u$ is a critical point of $I|_{\M }$.
	
	On the other hand, let $u \in \M $ be a critical point of $I|_{\M }$. We already know that $\langle I'(u), tu \rangle =t \langle I'(u), u \rangle=0 $ for all $t \in \R$, then to ensure that $ u$ is the critical point of $ I $ in $ \HA $ it is necessary to show that $\langle I'(u), v \rangle=0$ for all $v$ out of $ \R_u $ (space generated by $ u $). This is the same as showing that $T_u\M \perp \R_u$. For this, consider $v\in T_u\M$, then there is a way $\phi: [0,1] \subset \R \rightarrow \M$ such that $\phi(0)=u$ and $\phi'(0)=v.$ Note that
	\begin{eqnarray*}
		\langle I'(\phi(t)),\phi(t)\rangle= I''(\phi(t))\phi (t)\phi'(t)+ I'(\phi(t))\phi'(t).
	\end{eqnarray*}
	As $\phi(0)=u \in \M$, we have $\langle I'(\phi(0)),\phi(0)\rangle=0$. So by making $ t = 0 $ and multiplying the above equality by $ u $ we obtain
	\begin{eqnarray*}
		I''(\phi(0))\phi(0)\phi'(0)u+ I'(\phi(0))\phi'(0)u=0.
	\end{eqnarray*}
	Replacing $\phi(0)=u$ and $\phi'(0)=v$, we have
	\begin{eqnarray*}
		\langle I''(u),u \rangle \langle v, u \rangle+  \langle I'(u),u\rangle v =0.
	\end{eqnarray*} 
	See that $\langle I'(u),u\rangle =0$ and by (\ref{2.8}), we have $\langle\alpha'(u), u\rangle < 0$, then $\langle v,u \rangle=0$ for all $v\in T_u\M$, thus concluding the first part of the lemma.
	
	To show the second sentence of this lemma, consider $ \{u_n \} \subset \M $ a $(PS)_c$ sequence of $I$, that is, $I(u_n)=c$
	and 
	\begin{equation}\label{norma}
	||I'(u_n)||=\sup_{v\in H_A^1; ||v||_A=1} \langle I'(u_n),v \rangle \rightarrow 0 .
	\end{equation}
	First, note that $ I (u_n) = c $. In addition, we have
	\begin{equation}
	||I'|_{\M}(u_n)||=\sup_{w\in T_{u_n}\M; ||w||_A=1} \langle I'(u_n),w \rangle  ,
	\end{equation}
	as \cite[Definition 5.10]{willen}, which also goes to zero, since it is a particular case of (\ref{norma}), which concludes the outgoing.
	
	To show the other side of the statement consider $\{u_n  \} \subset \M $ a $(PS)_c$ sequence for $I|_{\M }$, that is, $$I|_{\M}(u_n)=c$$
	and 
	\begin{equation}
	||I'|_{\M}(u_n)||=\sup_{w\in T_{u_n}\M; ||w||_A=1} \langle I'(u_n),w \rangle \rightarrow 0 .
	\end{equation}
	In the same way as above we have $I'(u_n)=I(u_n) =c$. In addition, we have seen that $\HA=T_{u_n}\M \oplus \R_{u_n}$, so every $v\in \HA$ such that $||v||_A=1$ can be written as  $v=w+z$ with $w\in T_{u_n}\M$ and $z\in \R_{u_n} $. In this way, we have
	\begin{eqnarray}
	\nonumber	||I'(u_n)||&=&\sup_{v\in H^1_A; ||v||_A=1} \langle I'(u_n),v \rangle \\
	\label{soma} 	&=&\sup_{ ||w||_A=1} \langle I'(u_n),w \rangle + \sup_{ ||z||_A=1} \langle I'(u_n),z\rangle.
	\end{eqnarray}
	The first term of (\ref{soma}) converges to zero by hypothesis and the second term converges to zero by the same argument used in the first part of that demonstration.
\end{proof}

\section{ Preliminaries of Theorem \ref{teo2.1}}

To prove the theorem \ref{teo2.1} we need some auxiliary results. In order to facilitate the notation we will define the following functional
\begin{equation}\label{funcionalA}
A(u)=\int_{\R^N}a   (x)|u|^p dx
\end{equation}
and
\begin{equation}\label{funcionalB}
B(u)=\int_{\R^N}b   (x)|u|^q dx.
\end{equation}
In addition we will need the following definitions.

\begin{definition} 
We say that the functional $ F $ is weakly continuous when $F(u_n  ) \rightarrow  F(u)$ whenever $u_n \rightharpoonup u$, as $n \rightarrow \infty.$
\end{definition}

\begin{definition} 
	We say that the functional $ F': X \rightarrow X^* $ is completely continuous when $F'(u_n  ) \rightarrow  F'(u)$ whenever $u_n \rightharpoonup u$, as $n \rightarrow \infty.$
\end{definition}

\begin{lemma}\label{lema3.1}
	Suppose the hypothesis $ (D_1) $ is satisfied. So, $A' , B ' : \HA \rightarrow  \HA^*$ are completely continuous.
\end{lemma}

\begin{proof}
	We will begin by proving that $ A '$ is completely continuous. Let $ u_n  \in \HA$ be with $u_n \rightharpoonup u$. 
Being $ \{u_n \} $ bounded in $ \HA $, using the diamagnetic inequality we obtain
	\begin{equation} 
	\int|\nabla|u_n||^2\leq \int |\nabla_A u_n|^2<C\;\;
	\end{equation}
	for all $n\in \R$, whence $\{|u_n|\}$ is bounded in $H^1.$
	Turning to a subsequence if necessary, by Rellich-Kondrachov's theorem \cite[Theorem 8.16]{kavian} we have
	\begin{eqnarray}\label{3.1a}
	|u_n| &\rightharpoonup& u \mbox{ in } H^1(\R^N);\\
	\label{3.2}		|u_n| &\rightarrow & u \mbox{ in } L^l_{loc}(\R^N) \mbox{ for all }  2\leq l <2^* ;\\
	\label{3.3}			|u_n| &\rightarrow& u \mbox{ a.e. }  \R^N;
	\end{eqnarray}
	Choose $ v_n := |u_n  |^{p- 2} u_n  -  |u|^{p- 2} u$. By (\ref{3.3})
	\begin{equation}\label{zero}
	v_n \rightarrow 0 \;\mbox{ a.e. in } \R^N.
	\end{equation} 
	Knowing that $|a+b|^t \leq 2^{t-1}(a^t+b^t)$ for $t>1$ we have,
	\begin{eqnarray}
	\nonumber |v_n|^{\frac{p}{p-1}} &=&||u_n  |^{p- 2} u_n  -  |u|^{p- 2} u |^{\frac{p}{p-1}}  \\
	\nonumber  &\leq & 2^{\frac{p}{p-1}-1}[(|u_n  |^{p- 1}  )^{\frac{p}{p-1}}  -  (|u|^{p- 1})^{ \frac{p}{p-1}}]\\
	\label{limt}  &=& C(|u_n|^p-|u|^p).
	\end{eqnarray}
	Note that $ u \in \HA $, then $|u| \in H^1$, where
	\begin{equation}\label{finita}
	\left( \int(|u|^p)^{\frac{r}{r-1} }\right)^{\frac{r-1}{r}}= ||u||^p_{\frac{pr}{r-1}}< \infty,
	\end{equation}
	since by ($ D_4 $) we have $ \frac{pr}{r-1}<2 ^ * $. The same goes for every $ u_n $ of the given sequence.
	Hence, for (\ref{limt}) and (\ref{finita}) we have $|v_n|^{\frac{p}{p-1}} \in L^{\frac{r}{r-1}}$. 
	By boudedness of $\{u_n\}$ 
	in $ \HA$, by (\ref{finita}) and (\ref{zero}) we have $\{ |v_n|^{\frac{p}{p-1}} \}$ is bounded in $L^{\frac{r}{r-1}}(\R^N).$ This gives us that there is a subsequence such that
	\begin{equation}\label{3.6}
	|v_n|^{\frac{p}{p-1}}\rightharpoonup 0 \;\; \mbox{ in } L^{\frac{r}{r-1}}.
	\end{equation}
	Now, take $w\in \HA$ such that $||w||_A<1.$ Due to the inequalities H\"{o}lder and Sobolev we have
	\begin{eqnarray}
	\nonumber |\langle   A'(u_n) - A'(u), w \rangle |  & = &  \left| \int_{\R^N }a(x)v_n w dx \right| \leq \int_{\R^N } |a(x)|^\frac{1}{p}|w||a(x)|^\frac{1}{p'} v_n  dx        \\
	\nonumber   & \leq & \left(  \int_{\R^N } |a(x)| |w|^p dx \right)^\frac{1}{p} \left(  \int_{\R^N } |a(x)| |v_n|^{p'} dx \right)^\frac{1}{p'} \\ 
	\nonumber   & \leq &    ||a||_r^{\frac{1}{p}}     \left(  \int_{\R^N } |w|^{pr'} dx \right)^\frac{1}{pr'} \left(  \int_{\R^N } |a(x)| |v_n|^{p'} dx \right)^\frac{1}{p'}      \\
	\label{right}   & \leq & C ||a||_r^{\frac{1}{p}}     ||w||_{pr'}  \left(  \int_{\R^N } |a(x)| |v_n|^{p'} dx \right)^\frac{1}{p'},       
	\end{eqnarray}
	with $C>0$ constant. Then $a \in L^r=( L^{r'})^*$, with this and by (\ref{3.6}) we have (\ref{right}) goes to zero uniformly with respect to $||w||_A\leq 1$. Which proves that $ A '$ is completely continuous. For $ B '$ the proof is analogous.
\end{proof}





Assuming that the hypothesis $ (D_2) $ is satisfied, we are interested in showing that the functional satisfies the $ (PS)_c $ condition in Nehari manifold, for all $ c \in \R.$ Since the weight functions can change signal, we will separate the functional $ A $ and $ B $ previously defined in their positive and negative parts in order to show that the positive part 
of its derivatives are completely continuous.
In this way, we will make the following definitions.
\begin{equation}\label{3.7}
a^-  (x) := \max\{0, - a(x)\},  \;\;a^+ (x) := \max\{0, a(x)\},
\end{equation}
and we define $b^{\pm}  (x)$ similarly. Still, 
\begin{equation}\label{}
A_{\pm}  (u) := \int_{\R^N}a^{\pm}  (x)|u|^q dx, \;\;  B_{\pm}  (u) := \int_{\R^N} b^{\pm} (x)|u|^p dx.
\end{equation}
Then, we present the following result.

\begin{lemma}\label{lema3.2}
	Suppose the hypothesis $(D_2 )$ is satisfied. Then $A'_+ , B'_+ :\HA \rightarrow  \HA^*$ are completely continuous.
\end{lemma}

\begin{proof}
	First we show that $ A '_ + $ is completely continuous. Be $\{u_n\}\in \HA$ and $u_n \rightharpoonup u_0$ in $\HA$. If we need a subsequence and use the same argument of Lemma expression (\ref {lema3.1}) we get (\ref {3.1a}) - (\ref {3.3}). As we have done before, choose $ v_n := |u_n  |^{p- 2} u_n  -  |u|^{p- 2} u$. As $u_n \rightarrow u$ in $L^p_{loc}(\R^N)$, by result \cite[Theorem A.2]{willen} 
	and by (\ref{3.3})
	\begin{equation}\label{3.8}
	v_n\rightarrow 0 \mbox{  in } L^{\frac{p}{p-1}}_{loc}\R^N.
	\end{equation}
	By the hypothesis ($D_2$) for all $\varepsilon >0$ there will be an $ R> 0 $ such that
	\begin{equation}\label{3.9}
	a^+(x)<\varepsilon, \;\; \mbox{ whenever } |x|>R.
	\end{equation}
	Using the inequalities of H\"{o}lder and Sobolev and by (\ref {3.8}) we obtain
	\begin{eqnarray}
	\label{3.10}\sup_{||w||\leq 1} \left|  \int_{|x|\leq R}a^+ (x)v_n w dx\right|& \leq & ||a^+||_{\infty}\left( \int_{|x|\leq R}  |v_n |^{\frac{p}{p- 1}}\right)^{\frac{p- 1}{ p}}\left( \int_{ |x|\leq R}  |w|^p \right)^{\frac{1}{p}} \\
	\nonumber  &\leq& C_1\left(  \int_{|x|\leq R} |v_n |^{\frac{ p}{p- 1}}\right)^{\frac{p- 1}{ p}}\;\;  \rightarrow\; 0,\end{eqnarray}
	as $  n \rightarrow  \infty.$ As seen in the previous lemma, $ \{v_n \} $ is limited in $L^{\frac{ p}{p- 1}}(\R^N)$. Using this fact, the inequalities of H\"{o} lder and Sobolev and by (\ref {3.9}) we obtain that there exists a constant $ C_2> 0 $ independent of $ \varepsilon  > 0$ such that 
	\begin{equation}
	\label{3.11}\sup_{||w||\leq 1} \left|  \int_{|x|\leq R}a^+ (x)v_n w dx\right| \leq  C_2 \varepsilon .
	\end{equation} 
	Using (\ref{3.10}) and (\ref{3.11}), we have
	\begin{equation*}
	\sup_{||w||\leq 1}|\langle A'_+(u_n) -A'_+(u),w \rangle |= \sup_{||w||\leq 1} \left|  \int_{|x|>R}a^+ (x)v_n w dx\right| \;\;  \rightarrow\; 0,
	\end{equation*}
	as $  n \rightarrow  \infty,$ from which we conclude that $A'_+$ is completely continuous. For $ B '_ + $ the argument is analogous.
	
\end{proof}






\begin{lemma}\label{lema4.1}
	Suppose the hypotheses $ (D_4) $ or $ (D_5) $ are satisfied. Then, every $ (PS)_c $ sequence $ \{u_n \} \subset \M $ is bounded.
\end{lemma}
\begin{proof}
	Let $c \in \R$ and let $\{u_n  \} \subset \M$ a $(PS)_c$ sequence. Then
	\begin{equation}\label{3.12}
	||u_n||^2 = A(u_n  ) + B(u_n  ),
	\end{equation}
	and $I' (u_n  ) \rightarrow  0$, $I(u_n  ) \rightarrow  c$. 
	If  $(D_4 )$  is sitisfied, then $B(u_n  ) \geq 0$ and by (\ref{3.12}) and by the limitation of $I(u_n)$,
	\begin{eqnarray}\label{(3.13)}
	I(u_n  ) &=&   \frac{1}{2} ||u_n||^2_A - \frac{1}{p} A(u_n  ) - \frac{1}{q} B(u_n  )    \\
	\nonumber       &=&  \left( \frac{1}{2}-\frac{1}{p} \right)||u_n||^2_A +  \left( \frac{1}{p}-\frac{1}{q} \right) B(u_n  ) \geq \left( \frac{1}{2}-\frac{1}{p} \right)||u_n||^2_A,
	\end{eqnarray}
	for all $ n $ large enough. 
	Also, if $ (D_5) $ is satisfied, then $ A (u_n) \leq 0 $ and by (\ref {3.12}) again we get
	\begin{eqnarray}\label{(3.14)} 
	I(u_n  ) &=&   \frac{1}{2} ||u_n||^2_A - \frac{1}{p} A(u_n  ) - \frac{1}{q} B(u_n  )    \\
	\nonumber      &=&  \left( \frac{1}{2}-\frac{1}{p} \right)||u_n||^2_A  - \left( \frac{1}{p}-\frac{1}{q} \right) A(u_n  ) \geq \left( \frac{1}{2}-\frac{1}{p} \right)||u_n||^2_A,
	\end{eqnarray}
	for all $n$ large enough. As $ I (u_n) \rightarrow c $, we have that in the two cases $ \{u_n \} $ is a bounded sequence. As we wanted to demonstrate.
\end{proof}

Now we are ready to show that $ (PS) _c $ condition is satisfied by the functional $I$ in $\M $ for all $c \in  \R.$

\begin{prop}\label{prop3.3}
	Supose $(D_1 )$ or $(D_2 )$ and $(D_4 )$ or $(D_5 )$ are satisfied. Then, the functional $ I $ satisfies $(PS)_c$ condition in $\M $ for all $c \in  \R.$
\end{prop}

\begin{proof}	Let $c \in \R$ and let $\{u_n  \} \subset \M$ a $(PS)_c$ sequence. Since we are under the assumptions $ (D_4) $ or $ (D_5) $, by the Lemma \ref{lema4.1}, $ \{u_n \} $ is a bounded sequence. Thus, there is $ u \in \HA $ such that, passing to a subsequence if necessary, $ u_n \rightharpoonup u $. Thus, by having $ I '(u_n)\rightarrow 0 $, we obtain $ I' (u) = 0 $. Thereby,
	\begin{equation}\label{3.15}
	\langle I' (u_n  ) -  I' (u), u_n  -  u\rangle =||u_n  -  u||^2 -  \langle A' (u_n  ) -  A' (u), u_n  -  u\rangle  -  \langle B'(u_n  ) -  B'(u), u_n -  u\rangle  \rightarrow  0.
	\end{equation}
	Now, if $(D_1 )$ is satisfied, then by Lemma \ref{lema3.1}, $A' (u_n  ) \rightarrow  A' (u)$ and $B ' (u_n  ) \rightarrow  B'(u)$. Thus, by (\ref{3.15}) we obtain $u_n \rightarrow  u \in X$.
Suppose now that $ (D_2) $ is satisfied. Using the fact that the function  $v \mapsto  |v|^t$ is convex to $ t \geq 2 $ (in particular for $t = p$ and $q$), we get $(|v|^{t- 2} v -  |u|^{t- 2} u)(v -  u) \geq 0$.
	With this, by (\ref{3.15}),
	\begin{eqnarray*}
		||u_n  -  u||^2_A   &-& \langle A'_+ (u_n  ) -  A'_+ (u), u_n  -  u\rangle  -  \langle B'_+(u_n  ) -  B'_+(u), u_n -  u\rangle   \\
		&\leq&  \langle A'(u_n  ) -  A'(u), u_n  -  u\rangle  -  \langle B'(u_n  ) -  B'(u), u_n -  u\rangle  \rightarrow  0.
	\end{eqnarray*}
	By Lemma \ref{lema3.2}, $A'_+ (u_n  ) \rightarrow  A'_+ (u)$ and $B'_+  (u_n  ) \rightarrow  B'_+ (u)$, then $u_n \rightarrow  u $ also in this case, which concludes the proof of the proposition.
\end{proof}

To prove the Theorem \ref{teo2.1} we need to make some considerations.

\begin{definition} The set $ K \subset \HA $ is symmetric if $K = - K$. 
\end{definition}

\begin{definition} 
	Let
	$$\Sigma := \{K \subset X : K \mbox{ is closed and symmetric}\}.$$
	For $K \neq  \emptyset$ and $K \in \Sigma$, the Krasnoselskii genus of $ K $ is the smallest integer $ n $ such that there is an odd function $f \in  C(K, \R^n \setminus \{0\})$. 
	
	The $ K $ genus is denoted by $ \gamma (K) $. If there is no $ f $ that satisfies the above properties for any $ n $, then $ \gamma (K): = \infty $. Note also that $\gamma (\emptyset) := 0$.
\end{definition}

\begin{theorem}\label{teo3.4}
	\cite[Theorema II.5.7]{struwe}. Assume that $J \in C^1 (M )$ be a functional even in a manifold $C^{1,1}$, complete and symmetric $ M \subset V \setminus \{0\}$ in a Banach space $V$. Supose that $J$ satisfies the $(PS)_c$ condition for all $c \in \R$ and is bounded below in $ M $. Consider
	$$\hat{\gamma} (M ) := \sup\{\gamma (K ) : K \subset M \mbox{ is compact and simetric}\}.$$
	Then, the functional $ J $ has at least $ \hat{\gamma} (M) \leq \infty $ critical point pairs.
\end{theorem}

\subsection{ Proof of Theorem \ref{teo2.1}}
Our goal is to show that the $ \Pd $ problem has infinite solutions. For this, we are considering the functional $ I $ defined in the Nehari manifold $\M \subset \HA$.
By Lemmas \ref{lema2.2}-\ref{lema2.4} and by Proposition \ref{prop3.3}, $\M $ is a symmetrical and closed $C^2$ manifold, $I(u) > 0$ for all $u  \in \M$ and $I$ satisfies the $(PS)_c$ condition in $\M $ for all $c \in \R$. With this, we are in the hypothesis of the theorem \ref{teo3.4}. It remains then to show that $\hat{\gamma} (\M )= \infty$.
We will do this by proving that for all $ n \geq 1 $ there is a symmetric and compact set $K_n \subset \M$ such that $\gamma (K_n ) \geq n$. Hence, the first statement of this theorem follows from the Lemma \ref{lema2.5} and the Theorem \ref{teo3.4}.

Let $n \geq 1$ and let $X_n$ a subspace generated by $ n$ functions $v_j \in C_0^{\infty} (\R^N   )$ linearly independent and such that $\mbox{supp } v_j \subset \{x \in \R^N   : b(x) > 0\}$  and let
$$S^{ n-1} := X_n \cap \{u \in \HA : ||u||_A  = 1\}.$$
By the definition of $ S^{ n-1}$, we get $ B(u) > 0$ for all $u \in  S^{ n-1}$ and by Lemma \ref{lema2.2} the equation  $\alpha'_u  (t) = 0$ has exactly one solution $t_u \in (0, \infty)$. Thus, the application $\phi : S^{ n-1} \rightarrow  \M $ given by $\phi(u) := t_u u$ is well defined. Moreover, by the way we find $t_u$ in Lemma \ref{lema2.2} we can see that $t_u=t_{-u}$, then
\begin{equation*}
\phi(-u)=t_{-u}(-u)=-t_uu=-\phi(u)
\end{equation*} 
giving us $\phi$ is an odd application.

Affirmation: $\phi:u \mapsto  t_u $ is a continuous function.

To show the statement, note that if the necessary and sufficient condition of existence of a $ t_u $ given in Lemma \ref{lema2.2} is satisfied, then $T''_u  (t) < 0$ for $t = t_u$, as shown in the graph \ref{fig4}. Thus, by calling $f(t,u)=T'_u(t), $ we will have $f(t_u,u)=T'_u(t_u)=0, $ with $\frac{\partial f}{\partial t}=T''_u(t)<0$. Hence, by the implied function theorem we obtain the continuity of $ u \mapsto t_u $, concluding the statement.
Thus, $ \phi $ is a continuous and odd function of $ K_n $ into $S^{n- 1} $ and follows from the property of the genus that $\gamma (K_n ) = \gamma ( S^{ n- 1}  ) = n$, as \cite[Section II.5]{struwe}. 




\section{Preliminaries of Theorem \ref{teo2.2}}
We want to establish existence results and multiplicity of solutions for the case where the hypothesis $ (D_3) $ and one of the conditions $ (D_4) $ or $ (D_5) $ are satisfied.
We try to adapt to our case a method that was developed by \cite{16} and also used in \cite{jasz}.
Next, we present the necessary considerations to construct the proof of the Theorem \ref{teo2.2}.
Our next result shows the existence of nontrivial solutions to the problem \Pd  when $ a $ and $ b $ are periodic.

\begin{prop}\label{Prop4.2}
	Supose $(D_3)$ and also $(D_4)$ or $(D_5 )$ be satisfied. Then, there is $v \in \M$ such that $I'(v) = 0$ and $|v(x)| > 0$ for all $x \in \R^N$.
\end{prop}
\begin{proof}
	By Lemma \ref{lema2.2}, $I$ is bounded below in $ \M $. As a consequence of the Variational Principle of Ekeland \cite[Corolary A.3]{costa}, there exists a sequence $\{u_n  \} \subset \M$ such that
	$$I'(u_n  ) \rightarrow  0\;\;\; \mbox{ e }\;\;\; I(u_n  ) \rightarrow  c_0 := \inf_{u\in \M } I(u). $$
	By Lemma \ref{lema4.1}, we have $\{u_n  \}$ bounded. Thus, passing to a subsequence if necessary, will exist $u \in \HA$ such that $u_n \rightharpoonup u$. 
	By the principle of concentration and compactness, according to the Lemma of P.L. Lions \cite[Lemma 1.21]{willen}, if for some $ r> 0 $ we have
	$$\lim_{n\rightarrow \infty} \sup_{ y\in \R^N} \int_{B(y,r)} |u_n|^2 dx = 0,$$
	then $|u_n|  \rightarrow  0$ in $L^p (\R^N   )$ and $L^q (\R^N)$. 
In this case we would have 
	\begin{equation*} 
	A(u_n)=\int_{\R^N}a   (x)|u_n|^p dx\rightarrow  0
	\end{equation*}
	and
	\begin{equation*} 
	B(u_n)=\int_{\R^N}b   (x)|u_n|^q dx \rightarrow  0
	\end{equation*}
	and how $||u_n||^2_A =  A(u_n  ) + B(u_n  )$, it follows that $|u_n| \rightarrow  0$ in $\HA$. 
	
	However, by Lemma \ref{lema2.4} we have $||u||_A\geq \delta >0$ for all  $u\in \M$, which leads us to a contradiction.
	Thus, there exists ${y_n } \subset \Z^N $, $\rho  > 0$ and $R \geq r$ such that, passing to a subsequence if necessary we have  $v_n (x) := u_n  (x -  y_n )$ satisfying
	\begin{equation}\label{4.1}
	\lim_{n\rightarrow \infty} |v_n|^2 dx = \int_{B(y_n ,R)} |u_n|^2 dx  \geq \rho  > 0.
	\end{equation}
	Being $\{v_n\}$ bounded in $\HA$ there exists $v \in \HA $ such that
	$v_n  \rightharpoonup  v$ in $\HA $. In addition, by the diamagnetic inequality we have
	\begin{equation*} 
	\int|\nabla|v_n||^2\leq \int |\nabla_A v_n|^2<C\;\;
	\end{equation*}
	for all $n\in \R$, whence $\{u_n\}$ is bounded in $H_0^1.$ Moving on to a subsequence if necessary, by the Theorem of Rellich-Kondrachov \cite[Theorem 8.16]{kavian} we get  	
	\begin{eqnarray} 
	\label{4.2}    |v_n| &\rightharpoonup& |v| \;\; \mbox{ in } H^1(\R^N), \\
	\label{4.3}   |v_n| &\rightarrow &  |v| \;\; \mbox{in}\;\; L^l_{loc} (\R^N   ), 2 \leq  l < 2^*  ,\\
	\nonumber   |v_n| &\rightarrow & |v| \;\; \mbox{a.e. in}\;\; \R^N.
	\end{eqnarray}
	By (\ref{4.3}),
	$$ ||v||^2_2 \geq \int_{ B(0,R)} |v|^2 = \lim_{ n\rightarrow \infty} \int_{ B(0,R)} |v_n|^2 \geq \rho  > 0.$$
	Then $|v| \neq  0$. As $y_n \in \Z^N$, follows from the periodicity of $a$ that
	\begin{eqnarray*}
		A(v_n)&=&A(u_n(x-y_n))=\int a(x-y_n)|u_n(x-y_n)|^p\\
		&=&\int a(x) |u_n(x-y_n)|^p=\int a(x)|u_n(x)|^p=A(u_n).
	\end{eqnarray*}
	In the same way, by the $b$, $B(v_n)=B(u_n)$. Thereby, $||I'(v_n )||_A = ||I' (u_n  )||_A \rightarrow  0$. 
	We can show that $I'(v) = 0$. In fact, just take $\phi \in C_c^{\infty}(\R^N)$, using (\ref{4.1})-(\ref{4.3}) we get
	\begin{eqnarray}\label{Ifi}
	\langle I'(v_n), \phi \rangle \rightarrow 0.
	\end{eqnarray}
	On the other hand, 
	\begin{eqnarray*}
		\langle I'(v_n), \phi \rangle &=& \int \nabla_A u_n \nabla_A \phi -\int a(x)|v_n|^{p-1}\phi-\int b(x)|v_n|^{q-1}\phi  \rightarrow \langle I'(v), \phi \rangle,
	\end{eqnarray*}
	which together with (\ref{Ifi}) gives us that $\langle I'(v), \phi \rangle = 0$ for all $\phi \in C_c^{\infty}(\R^N). $ By the density of $C_c^{\infty}(\R^N)$ in $\HA$ we obtain $\langle I'(v), w \rangle = 0$ for all $ w \in \HA$, and we conclude that $I'(v)=0$.
	Now we will show that $ v $ is a minimum for $ I $ in $ \M. $ As $I(v_n ) = I(u_n  )$, $I(v_n ) \rightarrow  c_0$. 
	If $ (D_4) $ is satisfied, then $ b \geq 0 $ and by Fatou's Lemma
	\begin{eqnarray*}
		\liminf_{ n\rightarrow \infty} \left[ B(v_n )\right] \geq B( \liminf_{ n\rightarrow \infty}v_n  )=B(v).
	\end{eqnarray*}
	With this, we have
	\begin{eqnarray*}
		c_0&=& \liminf_{n\rightarrow \infty} I(v_n ) = \liminf_{n\rightarrow \infty }\left( I(v_n ) - \frac{1}{p}    \langle I'(v_n), v_n\rangle\right)\\
		& =& \liminf_{ n\rightarrow \infty} \left[ \left(  \frac{1}{2} -   \frac{1}{p} \right) ||v_n||_A^2 + \left(  \frac{1}{p} -  \frac{1}{q} \right) B(v_n )\right] \\
		&\geq &      \left(  \frac{1}{2} -   \frac{1}{p} \right) ||v ||_A^2 + \left(  \frac{1}{p} -  \frac{1}{q} \right) B(v  )\\
		&=&   I(v  ) - \frac{1}{p}    \langle I'(v ), v \rangle = I(v) \geq c_0.
	\end{eqnarray*}	
	Then, $I(v) = c_0$. 
Similarly, if $ (D_5) $ is satisfied, then $ a \leq 0 $ and hence
	\begin{eqnarray*}
		c_0&=& \liminf_{n\rightarrow \infty} I(v_n ) =  \liminf_{n\rightarrow \infty }\left( I(v_n ) - \frac{1}{p}    \langle I'(v_n), v_n\rangle\right)\\
		&=& \liminf_{ n\rightarrow \infty} \left[ \left(  \frac{1}{2} -   \frac{1}{p} \right) ||v_n||_A^2 + \left(  \frac{1}{p} -  \frac{1}{q}\right) A(v_n )\right] \\
		&\geq &      \left(  \frac{1}{2} -   \frac{1}{p} \right) ||v ||_A^2 + \left(  \frac{1}{p} -  \frac{1}{q} \right) A(v  )\\
		&=& I(v  ) - \frac{1}{p}    \langle I'(v ), v \rangle = I(v) \geq c_0,
	\end{eqnarray*}
thus, $ I (v) = c_0 $ also in this case.
	
\end{proof}

In the case where the hypothesis ($ D_3 $) is satisfied, we can not show the complete continuity of $ A'^+ $ and $ B'^+. $ Because of this it is not possible to guarantee the condition $ (PS)_c $ for the functional $ I$ in the range, for no $ c \in \R $. In order to overcome this problem, we need a type deformation argument. For this we will make use of the following notations
$$K := \{u \in \M  : I'(u) = 0\}$$
$$K_d := \{u \in K  : I(u) =  d\}. $$
In addition, we defined the following level sets of the Nehari manifold
\begin{eqnarray*} 
	I^d :=\{u \in \M  : I(u) \leq  d\}, & \;\;\;  I_e := \{u \in \M  : e \leq  I(u)\},&\;\;\;   I_e^d := I_e \cap I^d.
\end{eqnarray*}
Let $ \mathcal{K} $ be a subset of $ K $ such that $ \mathcal{K} = -\mathcal{K} $ and each orbit $ O (u) \subset K $ has a single $\mathcal{K}$.

Our goal now is to show that $ \mathcal {K} $ has infinite elements under the hypothesis $ (D_3) $ and also $ (D_4) $ or $ (D_5) $. For this we suppose that $ \mathcal {K} $ is finite to arrive at a contradiction. For the same argument used in \cite[Lemma 2.13]{16} we can show the next result.
\begin{lemma}\label{lema4.3}
The smallest of the distances between two distinct elements of the set $ K \cup \{0 \} $ is a positive number.
\end{lemma}
\begin{proof}
The demonstration of this lemma follows the idea of what was done in \cite[Lemma 2.13] {16}. 

We want to show that
	$\kappa := \inf\{||v -  w|| : v, w \in K  \cup \{0\}, v \neq  w\} > 0.$
	Take $v_n$ and $w_n$ in $\mathcal{K}$ and $k_n$, $l_n$ in $\Z^N$ such that $v_n(\cdot - k_n) \neq w_n(\cdot -l_n)$ for all $n$ and
	$$ ||v_n(\cdot - k_n) - w_n(\cdot -l_n)|| \rightarrow \kappa,\;\; \mbox{  as    } n\rightarrow \infty.$$
	
	Take $m_n=k_n-l_n.$ By the finiteness of $ \mathcal {K} $, passing a subsequence we have $ v_n = v \in \mathcal {K} $, $ w_n = w \in \mathcal {K} $. In addition we have two possibilities, or $ m_n = m \in \Z^N $ for almost everything $n$, or $|m_n| \rightarrow \infty.$
	If $m_n=m \in \Z^N$ for almost all $ n$, then
	\begin{eqnarray*}
		0&< &||v_n(\cdot - k_n) - w_n(\cdot -l_n)||=||v(\cdot - k_n) - w(\cdot -l_n)|| \\
		&=&||v-w(\cdot -m_n)||= ||v-w(\cdot -m)|| 	=\kappa ;\; \mbox{  for all   } n \in \N.
	\end{eqnarray*}
On the other hand, if $|m_n| \rightarrow \infty$, then $w(\cdot - m_n) \rightharpoonup 0$ and we have $\kappa = \lim_{n\rightarrow \infty} ||v-w(\cdot -m_n)||\geq ||v|| = 1. $
	As we wanted to demonstrate.
\end{proof}

In \cite{16} the minimum is assumed to be all $ v, w \in K $, but since $ 0 $ is an isolated critical point, $ \kappa $ remains positive even if $ v $ or $ w $ is $ 0 $ .

Next, we will establish a property that is related to the notion of Palais-Smale discrete attractor introduced in \cite{3}, also used in \cite[Lemma 4.4]{jasz} and \cite[Lemma 2.14]{16}.

\begin{lemma}\label{lema4.4}
	Assume $ (D_3) $ and also $ (D_4) $ or $ (D_5) $ are satisfied and that $\{u_n\}$,  $\{v_n \} \subset \M$ are two $(PS)_c$ sequences of $I$. Then, or $||u_n  -  v_n ||_A \rightarrow  0$ as $n \rightarrow \infty$ or $\limsup_{ n\rightarrow \infty} ||u_n  -  v_n ||_A \geq \kappa > 0$.
\end{lemma}

\begin{proof}
	It follows from Lemma \ref{lema4.1} that {$ u_n $} and {$ v_n$} are bounded in $\HA$. 
	
	\textbf{Case 1:} Suppose first that  $||u_n  -  v_n ||_p $, $||u_n  -  v_n ||_q \rightarrow  0$ as $n \rightarrow  \infty$. By H\"{o}lder inequality we have
	\begin{eqnarray*} 
		|| u_n  -  v_n||_A^2 &=& \langle I'(u_n  ), (u_n  -  v_n )\rangle -  \langle I'(v_n  ), (u_n  -  v_n )\rangle \\
		&+&  \int_{\R^N}a(x)[|u_n  |^{p- 2} u_n  -  |v_n |^{p- 2} v_n ](u_n  - v_n ) dx \\
		&+&  \int_{\R^N}b(x)[|u_n  |^{q- 2} u_n  -  |v_n |^{q- 2} v_n ](u_n  - v_n ) dx \\
		&\leq &  \langle I'(u_n  ), (u_n  -  v_n )\rangle -  \langle I'(v_n  ), (u_n  -  v_n )\rangle \\
		&+&  ||a||_{\infty}(||u_n  ||_p^{p- 1}   -  ||v_n||_p^{p- 1} )||u_n  - v_n||_p \\
		&+&  ||b||_{\infty}(||u_n  ||_q^{q- 1}   -  ||v_n||_q^{q- 1} )||u_n  - v_n||_q.
	\end{eqnarray*}
	First, as $I'(u_n  )\rightarrow 0$ and $I'(v_n  )\rightarrow 0$ and also {$u_n  $} and {$v_n $} are bounded in $\HA$, it follows that $\{u_n-v_n \}$ is also bounded in $\HA$, thereby 
	$$\langle I'(u_n  ), (u_n  -  v_n )\rangle \rightarrow 0 \;\;\mbox{ and }\;\;    \langle I'(v_n  ), (u_n  -  v_n )\rangle\rightarrow 0.$$
	In addition, by limitation of $\{u_n\}$ and $\{v_n \} $ in $L^p (\R^N )$ and $L^q (\R^N )$, we conclude that $||u_n  -  v_n ||_A\rightarrow 0$. 
	
	\textbf{Case 2:} Let us now assume that $||u_n  -  v_n ||_p \nrightarrow 0$ or $||u_n  -  v_n ||_q \nrightarrow 0$ as $n \rightarrow \infty$.
	As {$u_n  $} and {$v_n $} are bounded in $\HA$, by diamagnetic inequality {$|u_n|  $} and {$|v_n| $} are bounded in $H^1.$ Then, by P.L. Lions Lemma \cite[Lemma 1.21]{willen}, there exists $\delta_0 > 0, \{y_n \} \subset \Z^N$ and $r > 0$ such that, passing to a subsequence if necessary, $u_n  (x -  y_n ) -   v_n (x -  y_n )$ satisfies
	\begin{equation}\label{4.4}
	\lim_{ n\rightarrow \infty } \int_{ B(0,r)}|u_n  (x -  y_n ) -  v_n (x -  y_n )|^2 dx \geq \delta_0  > 0.
	\end{equation}
	Note that $I$ and $\M $ are invariant by translating $u \mapsto  u(\cdot -  k)$, $k \in \Z^N$, thereby defining
	$$u_n^1 (x) := u_n  (x -  y_n ) \mbox{ and } v_n^1 (x) := v_n (x -  y_n ),$$
	we have that $u_n^1 $, $v_n^1 \in \M$ and $\{u_n^1 \}$, $\{v_n^1 \}$ are $(PS)_c$ sequences (with the same $c$). 
	
With this, we are again in the hypotheses of the lemma \ref{lema4.1}, where we obtain that $\{u_n^1\} $, $\{v_n^1\}$ are bounded. Then, there exists $u^1 $ and $v^1 \in \HA $ such that,
	$$u_n^1 \rightarrow u^1   \mbox{ and } v_n^1\rightarrow v^1  ,$$
	as $n\rightarrow \infty$. Turning to a subsequence if necessary, (\ref{4.2}) and (\ref{4.3}) are also valid for $ { u_n^1 }$ and ${ v_n^1 }$. By (\ref{4.4}) and the strong convergence of $u_n$ and $v_n$ in $L^2_{loc} (\R^N   )$, we have $u^1 -  v^1 \neq  0$. 
	As was previously seen $I' (u^1 ) = I' (v^1 ) = 0$. Thus, $u^1$, $v^1 \in K  \cup\{0\}$ and hence	
	\begin{equation*}
	\limsup_{n\rightarrow \infty} ||v_n -  u_n||_A \geq \liminf_{n\rightarrow \infty } ||v_n -  u_n||_A \geq ||v^1 -  u^1||_A \geq \kappa,
	\end{equation*}
	which completes the proof.
\end{proof}

Recalling our notation, we are denoting the inner product in $ \HA $ for $ \langle \cdot, \cdot \rangle $. Define the $ I $ gradient by duality, that is, by the set
$$\langle\nabla_A I(v), w \rangle := \langle I'(v), w\rangle \mbox{ for all } w \in \HA.$$
Since $ \M $ is a $ C^2 $ manifold, closed in $ \HA $, by the result \cite[Lemma II.3.9]{struwe}, we have  $I|_{\M}$ admits a vector field of pseudo gradients $ H $, that is, a locally Lipschitzian and continuous application $H : \M  \setminus T \rightarrow  T \M $ such that
\begin{equation}\label{H1} 
||H(v)||_A < 2||\nabla_A I(v)||_A,
\end{equation}
\begin{equation}\label{H2} 
(H(v), \nabla_A I(v)) > \frac{1}{2} ||\nabla_A I(v)||_A^2
\end{equation}
worth for all $ v \in \M \setminus K $. Moreover, since $ I $ is even, we assume that $ H $ is odd, as can be seen in \cite[Remark II.3.10]{struwe}.
Note that $\langle\nabla_A I(v), v\rangle = 0$ if $v \in  \M$ and $\nabla_A I$ is equal to the gradient of $ I|_{\M}$ for each $ v$.  

Now, let $\eta : \mathcal{D} \rightarrow  \M$  be the flow corresponding to the field of pseudo-gradient vectors $H$, that is, $\eta$ is defined by
$$\label{P}
\left\{ \begin{array} [c]{ll}
\frac{d}{dt}  \eta(t, v) = - H(\eta(t, v)),&\\
\eta(0, v) = v ,&\\
\end{array}          \right.\leqno {( P) }
$$
Here $\mathcal{D} := \{(t, v) : v \in \M  \setminus K, t \in I_v \}$ and $I_v := (T^-  (v), T^+ (v))$ it is the maximum interval of existence for the initial value problem $(P)$.
\begin{remark}
For a result in \cite[Theorem A.4]{Rab} $ \eta $ is odd in $ v $. 
	
\end{remark}
\begin{remark}
	As $I \in C^2  (\M   )$, we can actually choose $ H $ as the gradient vector field of $ I|_{\M}$, that is, we can put  $H(v) := \nabla_A I(v),$ with $ v \in \M$. For this $ H $ we can show that the flow $ \eta $ exists for all $(t, v) \in \R \times \M$.
\end{remark}

\begin{lemma}\label{lema4.6}
	For all $v \in \M$ the limit $\lim_{ t\rightarrow  T^+ (v)} \eta(t,  v)$ exists and is a critical point of $I$.
\end{lemma}
\begin{proof} The proof follows similarly to what was done in \cite[Lemma 2.15]{16}, with $ \rho (d) $ replaced by $ \kappa $. Note that the argument used in \cite{16} only uses the existence of the pseudo-gradient flow $ \eta $ in a complete manifold and also the fact that the $ (PS)_c $ sequence is discrete. This last property is valid in the context of Lemma \ref{lema4.4}. 
	
	Let $v \in \M$ and let $ I(v) =D$. We will split the demonstration into two cases.
	
	\textbf{Case 1: } $ T^+(v)<+\infty $. For $0\leq s<t<T^+(v), $ by (\ref{H1}), (\ref{H2}) and $(2)$ we have that
	\begin{eqnarray*}
		||\eta(t,v )-\eta(s,v )   ||_A &\leq& \int_{s}^{t}	||H(\eta(\tau,v ))||_A d\tau\leq 2 \sqrt{2}\int_{s}^{t}\sqrt{ \langle  H(\eta(\tau,v )), \nabla_A I(\eta(\tau,v ) ) \rangle } d\tau\\
		&\leq & 2 \sqrt{2(t-s)}\left(\int_{s}^{t} \langle  H(\eta(\tau,v )), \nabla_A I(\eta(\tau,v ) ) \rangle  d\tau\right)^{1/2}\\
		&=& 2 \sqrt{2(t-s)}[I(\eta(s,v ))-I (\eta(t,v ) ) ]^{1/2}\leq 2 \sqrt{2(t-s)}[I(v)-c]^{1/2}.
	\end{eqnarray*}
	We then have the limit  $\lim_{ t\rightarrow  T^+ (v)} \eta(t,  v)$ exists, since $ T^+(v)<+\infty $. Also, the limit is a critical point of $I,$ otherwise we would have that the trajectory $t\mapsto \eta(t,  v)$ could continue beyond $ T^+(v)<+\infty $.
	
	\textbf{Case 2: } $ T^+(v)=+\infty $. We need to show that for all $\epsilon>0$ exists $t_{\epsilon}>0$ with $||\eta(t_{\epsilon}, v)- \eta(t, v)||<\epsilon$ for $t>t_{\epsilon}$. Supposing it is absurd that this is false. Thus we will have $ \epsilon $ between $ 0 $ and $ \kappa $, where $ \kappa $ is as in Lemma \ref{lema4.3}, and a sequence $(t_n)\subset \lbrack 0, \infty )$ with $t_n \rightarrow \infty$ and $||\eta(t_{n}, v)- \eta(t_{n+1}, v)||=\epsilon$.
	Choose the smallest $t_n^1\in (t_n,t_{n+1})$ such that $||\eta(t_{n}, v)- \eta(t_n^1, v)||_A=\frac{\epsilon}{3}$ and let $\delta_n:= \min_{s\in [t_n,t_n^1]}||\nabla_A I(\eta(s,v))||ds.$ Thus,
	\begin{eqnarray*}
		\frac{\epsilon}{3}&=& ||\eta(t_{n}, v)- \eta(t_n^1, v)||_A\leq \int_{t_n}^{t_n^1}||H(\eta(s,v))||ds\leq 2\int_{t_n}^{t_n^1}||\nabla_A I(\eta(s,v))||ds\\
		&\leq& \frac{2}{\delta_n} \int_{t_n}^{t_n^1}||\nabla_A I(\eta(s,v))||ds \leq \frac{4}{\delta_n} \int_{t_n}^{t_n^1} \langle H(\eta(s,v)),  \nabla_A I(\eta(s,v)) \rangle ds\\
		&=& \frac{4}{\delta_n} (I(\eta(t_n,v))- I(\eta(t_n^1,v)) ).
	\end{eqnarray*}
	But the latter term goes to zero as $ n \rightarrow \infty. $ Implies that $\delta_n \rightarrow 0$ and also, exists $s_n^1 \in [t_n,t_n^1] $ such that $\nabla_AI(\eta(s_n^1,v))\rightarrow 0 .$
In the same way, we seek the greatest $t_n^2\in (t_n^1,t_{n+1})$ for which $||\eta(t_{n+1}, v)- \eta(t_n^2, v)||_A=\frac{\epsilon}{3}$, and then $\nabla_AI(\eta(s_n^2,v))\rightarrow 0 .$
	Calling $v_n^1:=\eta(s_n^1,v)$ and $v_n^2:=\eta(s_n^2,v)$ we have $||v_n^1- \eta(t_n, v)||_A \leq \frac{\epsilon}{3}$ and $||v_n^2- \eta(t_{n+1}, v)||_A \leq \frac{\epsilon}{3}$, there is, $\{v_n^1\}$ and $\{v_n^2\}$ are two PS sequences such that
	$$\frac{\epsilon}{3} \leq ||v_n^1-  v_n^2  ||_A \leq 2\epsilon <\kappa,   $$
	which contradicts the Lemma \ref{lema4.4}.
	Thus we can show that for all $\epsilon>0$ existe $t_{\epsilon}>0$ with $||\eta(t_{\epsilon}, v)- \eta(t, v)||_A<\epsilon$ for $t>t_{\epsilon}$, therefore, the limit exists and is a critical point of $ I $.
\end{proof}

Let $O \subset \M$ and $\delta > 0$. Define
$$U_{\delta} (O) := \{w \in  \M  : dist(w,O ) < \delta \}.$$

\begin{lemma}\label{lema4.7}
	Let $d \geq c_0 = \inf_{ u\in \M}  I(u)$. Then, for all $\delta > 0$ exists $\epsilon  = \epsilon (\delta) > 0$ such that
	\begin{description}
		\item[$ (a)$  ]   $I_{d-  \epsilon }^{d+\epsilon} \cap K = K_d$;
		\item[  $(b)$ ]   $\lim_{ t\rightarrow T^+ (v)} I(\eta(t, v)) < d -  \epsilon$ for $ v \in I^{ d+\epsilon}  \setminus U_{\delta} (K_d )$.
	\end{description}

\end{lemma}
\begin{proof} 	
	\begin{description}
		\item[ $(a)$  ]    It follows immediately from the finiteness of $\mathcal{K}$. 
		\item[$(b)$   ]    This part can be proved by the same argument used in \cite[Lemma 2.16]{16}, but with $ \kappa $ instead of $ \rho (d + 1) $. This argument is based on the lemmas \ref{lema4.4} and \ref{lema4.6} and involves a careful analysis of the flow.
	\end{description}

\end{proof}

\subsection{ Proof of Theorem \ref{teo2.2}}

The existence of solutions was shown in Proposition \ref{Prop4.2}. To show the fact that there are infinitely many geometrically distinct solutions, we use the same argument used in \cite[Theorem 1.2]{16} as we will show below.

Consider the sequence
$$c_k := \inf\{d \in \R : \gamma (I^d ) \geq k\}, \;\;\; k \in N,$$
that is, for every $ k \in \N $ we take the lowest level $ d $ such that the genus of $ I^d $ is greater than $ k $.
We want to use the finiteness of $ \mathcal{K} $ to arrive at a contradiction. For this, we will show that
\begin{equation*}
K_{c_k} \neq  \emptyset \mbox{ and } c_k < c_{k+1} \mbox{ for all } k.
\end{equation*}
$K_{c_k} \neq  \emptyset$ assures us that for each term of the sequence there is a $w\in \HA$ such that $I'(w)=0$, that is, such that $ w $ is the solution of \Pd  and $ I(w) = c_k $. Also, showing that $ c_k <c_{k + 1}$ ensures that every $ k $ we are working at a different level from the previous one.
Let is make $ d = c_k $. By the Lemma \ref{lema4.3}, $\gamma(K_d)=0$ or $1.$ Using the continuity property of the genus, there is a $\delta$ such that $0 < \delta < \frac{\kappa}{2}$ and with $\gamma (\bar{U }) = \gamma (K_d )$, where  
$$U := U_{\delta} (K_d )= \{w \in  \M  : dist(w,K_d ) < \delta \},$$
that is, there is a range around the $ d $ level such that the genus remains the same.
For this $ \delta $, we choose one $\epsilon   = \epsilon (\delta) > 0$ such that the conclusion of Lemma \ref{lema4.7} follows. Thus, for each $v \in I^{d+\epsilon}  \setminus U$ exists $t \in \lbrack 0, T^+ (v))$ such that $I(\eta(t, v)) < d- \epsilon$.

Now, define the application $e: I^{d+\epsilon} \setminus U \rightarrow \lbrack 0, \infty)$;
$$  e(v):= \inf \{  t \in \lbrack 0, T^+ (v)); I(\eta(t, v)) < d- \epsilon  \} $$
As $d -   \epsilon$ is not a critical value of $ I $, it follows by the Lemma \ref{lema4.7} that the application $ e $ is continuous and even more, it is even (since $ I $ is even).

Let $h: I^{d+\epsilon} \setminus U \rightarrow I^{d-\epsilon}$; $h(v) := \eta(e(v), v)$. See that 
$$h(-v)=\eta(e(-v), -v)=\eta(e(v), -v)=- \eta(e(v),v)=-h(v),$$
There is, $h$ is odd and continuous.
Now, using the properties of the genus and the definition of $ c_k $, we get
$$\gamma (I^{ d+\epsilon}  ) \leq  \gamma (\bar{U} ) + \gamma  (I^{ d- \epsilon } ) \leq  \gamma (K_d ) + k -  1.$$
By the definitions of $ d = c_k $ and $ c_ {k + 1} $ we have
$$\gamma(K_d)\geq 1,\;\; \mbox{  if  } \;\; c_{k+1}>c_k$$
and
$$\gamma(K_d)> 1,\;\; \mbox{  if  } \;\; c_{k+1}=c_k.$$
By Lemma \ref{lema4.3}, we have $\gamma (K_d ) \leq  1$. Then we get $\gamma (K_d ) = 1$ and $K_d \neq  \emptyset$ and also $c_k < c_{k+1}$ for all $k$. Thus, there are infinite pairs $(\pm v_k)$ of solutions geometrically distinct from \Pd  such that $ I (v_k) = c_k $ contradicting the finiteness of $ \mathcal{K}, $ which completes the proof.

\vspace{0.6cm}

\bibliographystyle{unsrt}

\begin{thebibliography}{1}
\bibitem{2} 
Alama,S. and  Tarantello,G., 
\newblock Elliptic problems with nonlinearities indefinite in sign. 
\newblock In {\em J. Func. Anal.}, 141, pages 159-215, 1996.
	
\bibitem{1}  
Alama,S. and  Tarantello,G., 
\newblock On semilinear elliptic equations with indefinite nonlinearities. 
\newblock In {\em Calc. Var. PDE 1}, pages 439-475, 1993.

\bibitem{ClauFig}
Alves, C. O. and Figueiredo, G.M. 
\newblock Multiple solutions for a semilinear elliptic equation with critical growth and magnetic field. 
\newblock In {\em Milan Journal of Mathematics}, 82.2, pages 389-405, 2014.

\bibitem{alves.furtado} 
Alves,C.O.,  Figueiredo,G.M. and Furtado,M.F., 
\newblock On the number of solutions of NLS equations with magnetics fields in expanding domains. 
\newblock In {\em J. Differential Equations}, 251.9, pages 2534-2548, 2011.

\bibitem{ArSz} 
Arioli,G. and Szulkin,A., 
\newblock A semilinear Schrödinger equation in the presence of a magnetic field. 
\newblock In {\em Arch. Ration. Mech. Anal.}, 170.4, pages 277-295, 2003.

\bibitem{3} 
Bartsch,T. and  Ding,Y.H., 
\newblock  On a nonlinear Schr\"{o}dinger equation with periodic potential. 
\newblock In {\em Math. Ann.}, 313, pages 15-37, 1999.

\bibitem{4}  
Berestycki,H.I., Capuzzo-Dolcetta and  Nirenberg,L. 
\newblock Variational methods for indefinite superlinear homogeneous elliptic problems. 
\newblock In {\em Nonl. Diff. Eq. Appl.}, 2, pages 553-572, 1995.

\bibitem{BW} 
Brown, K.J., and Wu,T.F. 
\newblock A fibering map approach to a semilinear elliptic boundary value problem.
\newblock In {\em J. Differential Equations}, 2007, pages 1-9, 2007.

\bibitem{BZ2003} 
Brown,K.J. and Zhang,Y.,  
\newblock The Nehari monifold for a semilinear elliptic problem with a sign changing weight function.  
\newblock In {\em J. Differential Equations}, 193, pages 481-499, 2003.

\bibitem{ChabSzul} 
Chabrowski,J. and Szulkin,A., 
\newblock On the Schrödinger equation involving a critical Sobolev exponent and magnetic field. 
\newblock In {\em Topol. Meth. Nonl. Anal.}, 25, pages 3-21, 2005.

\bibitem{Cingolani2} 
Cingolani,S., Jeanjean,L. and Secchi,S., 
\newblock Multi-peak solutions for magnetic NLS equations without non-degeneracy conditions.
\newblock In {\em ESAIM Control Optim. Calc. Var.}, 15.3, pages 653-675, 2009.

\bibitem{costa} 
Costa, D. G., 
\newblock  An invitation to variational methods in differential equations. 
\newblock In {\em Springer Science and Business Media, 2010.

\bibitem{DP} 
Drabek,P. and Pohozaev,S.I.,  
\newblock Positive solutions for the p-Laplacian: application of the fibering method. 
\newblock In {\em Proc. Roy. Soc. Edinburgh} Sect. A, 127.4, pages 703-726, 1997.

\bibitem{EstLions} 
Esteban,M.J. and Lions,P.L., 
\newblock Stationary solutions of nonlinear Schrödinger equations with an external magnetic field.
\newblock In {\em  PDE and Calculus of Variations, in honor of E. De Giorgi, Birkhauser}, pages 401-449, 1990.

\bibitem{jasz} 
Jalilian,Y. and Szulkin,A., 
\newblock Infinitely many solutions for semilinear elliptic problems with sign-changing weight functions.
\newblock In {\em Appl. Anal.}, 93(4), pages 756-770, 2014. 
	
\bibitem{12}  
Kajikiya,R., 
\newblock Mountain pass theorem in ordered Banach spaces and its applications to semilinear elliptic equations.
\newblock In {\em Nonl. Diff. Eq. Appl.}, 19, pages 159-175, 2012.
	
\bibitem{kavian}  
Kavian,O., 
\newblock Introduction à la théorie des points critiques et applications aux problems elliptiques. 
\newblock In {\em Springer-Verlag, New York, Berlin}, 1994.
	
\bibitem{Kurata} 
Kurata,K., 
\newblock Existence and semi-classical limit of the least energy solution to a nonlinear Schrödinger equation with electromagnetic fields.
\newblock In {\em Nonlinear Anal.}, 41, pages 763-778, 2000.

\bibitem{olimpio1} 
Miyagaki,O.H., 
\newblock On a class of semilinear elliptic problems in $\R^N$ with critical growth. 
\newblock In {\em Nonlinear Anal.: Theory, Methods and Applications}, 29.7, pages 773-781, 1997. 
	
\bibitem{Rab} 
Rabinowitz,P.H., 
\newblock Minimax methods in critical point theory with applications to differential equations. 
\newblock In {\em  American Mathematical Soc.},No. 65., 1986.
	
\bibitem{struwe} 
Struwe,M.,  
\newblock Variational methods. Applications to nonlinear partial differential equations and hamiltonian systems.
\newblock In {\em  Springer-Verlag, Berlin}, 1990.
	
\bibitem{16} 
Szulkin, A. and  Weth,T., 
\newblock Ground state solutions for some indefinite variational problems. 
\newblock In {\em J. Func. Anal.}, 257, pages 3802-3822, 2009.
	  
\bibitem{willen} 
Willem,M., 
\newblock Minimax Theorems.
\newblock In {\em Birkhäuser, Basel}, 1996.

\bibitem{19} 
Wu,T.F., 
\newblock Multiplicity results for a semi-linear elliptic  equation involving sign-changing weight function}, 
\newblock In {\em Rocky Mountain J. Math.} 39, pages 995-1011, 2009.


\end{thebibliography}

\end{document}